\documentclass{gen-p-l}
\usepackage{textcomp}
\usepackage{latexsym}
\usepackage{amsfonts}
\usepackage{amssymb}
\usepackage{amsmath}
\usepackage{multirow}
\usepackage{mathrsfs}
\usepackage{wasysym}
\usepackage{stmaryrd}
\usepackage{esint}
\usepackage{rawfonts}
\input{prepictex}
\input{pictex}
\input{postpictex}
\DeclareMathAlphabet{\zap}{OT1}{pzc}{m}{it}
\def\be{\begin{equation}}
\def\ee{\end{equation}}
\def\bea{\begin{eqnarray*}}
\def\eea{\end{eqnarray*}}
\def\vol{~d\mu}
\newtheorem{main}{Theorem}

\DeclareMathOperator{\Aut}{Aut}

\DeclareMathOperator{\Int}{Int}
\DeclareMathOperator{\grad}{grad}

\numberwithin{equation}{section}

\newtheorem{thm}{Theorem}[section]
\newtheorem{lem}[thm]{Lemma}
\newtheorem{prop}[thm]{Proposition}
\newtheorem{cor}[thm]{Corollary}

\def\umu{\mbox{\textmu}}
\newenvironment{proofchiaro}{\medskip
{\sc Proof of Theorem \ref{chiaro}.}}{\hfill $\square$ 
\\}

\newenvironment{xpl}{\mbox{ }\\ {\bf  Example}\mbox{ }}{
\hfill $\diamondsuit$\mbox{}\bigskip}

\def\CC{\mathbb C}

\def\ZZ{{\mathbb Z}}
\def\RR{{\mathbb R}}
\def\QQ{{\mathbb Q}}
\def\CP{{\mathbb C \mathbb P}}
\def\proj{\mbox{\thorn}}

\begin{document}

\title{Calabi Energies of  Extremal Toric Surfaces}
%

\author{Claude LeBrun}
\address{Department of Mathematics\\
SUNY at Stony Brook\\
Stony Brook, NY 11794-3651}
\email{claude@math.sunysb.edu}
\thanks{Supported 
in part by    NSF grants DMS-0905159 and DMS-1205953.}

\date{}
\maketitle

 \begin{abstract}	
 We derive a formula for the $L^2$ norm of 
 the scalar curvature of any extremal  K\"ahler metric
 on a compact toric   manifold, stated purely in terms of the 
 geometry of the corresponding moment polytope. 
 The main  interest of this formula pertains
 to the case of complex dimension $2$, where it plays a 
 key role in construction of  Bach-flat metrics
on appropriate $4$-manifolds.     \end{abstract}

\section{Introduction}

In an audacious  attempt to endow  complex algebraic varieties with 
 canonical Riemannian metrics,  Eugenio Calabi \cite{calabix}
initiated  a systematic study of the squared $L^2$-norm
\begin{equation}\label{scalar}
{\mathcal C}(g) =  \int_M s^2~d\mu
\end{equation}
of the scalar curvature, considered as a functional 
on the space of K\"ahler metrics $g$ on a given compact complex manifold $(M,J)$;
here $s$ and $d\mu$ of course  denote the scalar curvature and Riemannian 
volume form of the given metric  $g$.  Given a 
 K\"ahler class
$\Omega \in H^{1,1} (M,\RR)\subset 
H^{2} (M,\RR)$, his aim was to minimize the functional ${\mathcal C}(g)$ among all 
K\"ahler metrics $g=\omega(\cdot , J\cdot)$ with K\"ahler class  $[\omega ]=\Omega$. 
Calabi showed that the   Euler-Lagrange equation for this variational problem
is equivalent to requiring that
$\nabla^{1,0}s$ be a holomorphic vector field, and he introduced the 
terminology  {\em extremal K\"ahler metrics}  for the  solutions of  this equation. 
It was  later shown \cite{xxel}  that 
any extremal K\"ahler metric on a compact complex manifold
actually minimizes the Calabi energy (\ref{scalar})  in its K\"ahler class. Moreover, 
when such a   minimizer exists, it is  actually unique in its
K\"ahler class, modulo automorphisms of the complex manifold 
\cite{xxgang2,donaldsonk1,mabuniq}. Our knowledge of 
existence remains imperfect, but  considerable 
progress \cite{arpasing,chenlisheng,dontor} has recently been made in the toric case that is focus 
of the present paper. However, a relatively elementary argument \cite{ls2} shows  that
the set of K\"ahler classes represented by extremal K\"ahler metrics on 
a compact complex manifold 
$(M, J)$ is necessarily 
 open   in $H^{1,1} (M,\RR)$. 

Rather than minimizing the squared $L^2$-norm of the scalar curvature, as in (\ref{scalar}), 
one might be tempted to instead minimize the squared $L^2$-norm of, say, the Riemann curvature
tensor or the Ricci tensor. However, 
Calabi also observed \cite{calabix} that, after appropriate normalization, 
 such functionals only differ from (\ref{scalar}) by a constant depending 
 on the K\"ahler class. In this respect, real dimension four occupies 
 a privileged position; not only does (\ref{scalar}) become scale invariant 
 in this dimension, but the relevant constants  only depend on the topology of 
 $M^4$, and so are independent of the K\"ahler class in question. For example, 
 the Riemann curvature $\mathcal R$ and the Ricci tensor $r$ satisfy
 \begin{eqnarray*}
\int_M |\mathcal R |^2d\mu &=& -8\pi^2( \chi +3\tau ) (M) + \frac{1}{4}{\mathcal C}(g)
\\ 
\int_M | r |^2d\mu &=& -8\pi^2( 2\chi +3\tau ) (M) + \frac{1}{2}{\mathcal C}(g)
\end{eqnarray*}
for any compact K\"ahler manifold $(M,g,J)$ of complex dimension $2$,
where $\chi(M)$ and $\tau (M)$ are respectively the Euler characteristic and
signature of the compact oriented $4$-manifold $M$. 
Similarly, the Weyl curvature $W$, which is the conformally invariant part  
of the Riemann tensor $\mathcal R$, satisfies 
\begin{equation}\label{weyl}
\int_M | W |^2d\mu = -12\pi^2 \tau (M) +  \frac{1}{12}{\mathcal C}(g)
\end{equation}
for any compact K\"ahler surface $(M,g,J)$. Thus, 
if a K\"ahler metric $g$ on  $M^4$ 
 is a critical point of any of these Riemannian functionals, considered  as 
 a function on the bigger space of all Riemannian metrics on $M$, it must, in particular,  be 
 an {\em extremal} K\"ahler metric. In connection with (\ref{weyl}), 
this observation has interesting consequences, some of which will be touched on 
in this article.

The primary goal of this article is to  calculate the Calabi energy of any 
extremal K\"ahler metric on any {\em  toric surface} --- that is, 
on any simply connected compact complex manifold of complex dimension two   
which carries  a compatible effective action of the $2$-torus $T^2=S^1\times S^1$. Any 
K\"ahler class on a toric surface is represented by a $T^2$-invariant K\"ahler metric, and, relative to such a metric, the action  is  generated by two periodic Hamiltonian vector fields. 
This pair of Hamiltonians gives us an  $\RR^2$-valued 
moment map, under which the image of our complex surface 
is a convex polygon $P\subset \RR^2$. Moreover, modulo    translations and  $SL(2, \ZZ)$  transformations, 
the moment polygon $P$  only depends on the  given the K\"ahler class.
Euclidean area measure on the 
interior of $P$ then allows us to define a barycenter for $P$  and a moment-of-inertia
matrix $\Pi$ of $P$ relative to this barycenter. 
The edges of $P$ have rational slope,  and  are  therefore endowed 
with  preferred rescalings  $d\lambda$ of $1$-dimensional Lebesgue measure,
chosen so that    intervals of unit length correspond to  separation
 vectors  which are indivisible  elements of the 
 integer lattice $\ZZ^2$. This allows us to also define a barycenter of
 the perimeter of  $P$, and hence also   a vector $\vec{\mathfrak D}\in \RR^2$
 connecting the barycenter of the interior to the barycenter of the perimeter. 
 Combining these ingredients, we then obtain a convenient formula
 for the Calabi energy of any extremal toric surface: 
 
\begin{main} \label{behold}
 Let $(M, J, \Omega)$ be a toric surface with fixed K\"ahler class, and let
 $P$ be the associated moment polygon. Then 
 any K\"ahler metric $g$ with K\"ahler form $\omega \in \Omega$  has
scalar curvature $s$ satisfying 
$$
\frac{1}{32\pi^2}\int_M s^2 d\mu_g \geq 
\frac{|\partial P|^2}{2}\Big( \frac{1}{|P|~} + \vec{\mathfrak D}\cdot  {\Pi}^{-1} \vec{\mathfrak D}
\Big)
$$
with 
equality iff $g$ is an extremal K\"ahler metric. Here $|P|$ denotes the area of 
the interior of  $P$, $|\partial P|$ is the $\lambda$-length of its boundary,  
 $\Pi$ is the moment-of-inertia
matrix of $P$, and $\vec{\mathfrak D}$ is the vector joining  the barycenter
$P$ to  the barycenter of $\partial P$. 
\end{main}

We give two proofs of this result. Our first proof, which is specifically adapted to 
complex dimension $2$, can be found in \S \ref{focus} below. Then, in \S \ref{abreu}, we  
prove a generalization, Theorem \ref{behind},  which holds for  
toric manifolds of  arbitrary complex dimension. 
 However, both proofs crucially depend on a detailed understanding  of both 
the Futaki invariant and  toric manifolds. 
We have therefore found it useful to preface our main calculations with 
a careful exploration  of the underpinnings  of these ideas. The article then 
concludes with a discussion of 
examples that illustrate  our current knowledge   of Bach-flat K\"ahler metrics.

\section{The Futaki Invariant}

If  $(M^{2m},J)$ is a compact complex $m$-manifold of K\"ahler type, 
and if  $${\mathfrak h}=H^0(M, {\mathcal O}(T^{1,0}M))$$ is
the associated Lie algebra of holomorphic fields on $M$, 
the {\em Futaki invariant} assigns an element ${\mathfrak F} (\Omega)$ of  the Lie coalgebra 
${\mathfrak h}^*$ to every 
 K\"ahler class
$\Omega$ on $(M,J)$. 
To   construct this element, let  $g$ be a K\"ahler metric, 
 with K\"ahler class  $[\omega] = \Omega$, scalar curvature $s$, Green's operator ${\zap G}$, and volume form $d\mu$. We   then define the Futaki invariant 
$${\mathfrak F} (\Omega) :  H^0(M, {\mathcal O}(T^{1,0}M))\longrightarrow \CC$$
to be the linear functional 
$$\Xi \longmapsto -2 \int_M \Xi ({\zap G}s)\, d\mu\, . $$
It is a remarkable fact, due to Futaki \cite{futaki0}, Bando \cite{bando}, 
and Calabi \cite{calabix2}, that 
${\mathfrak F} (\Omega)$ only depends on the K\"ahler class $\Omega$, and not
on the particular metric $g$ chosen to represent it. 

We will now assume henceforth  that $b_1(M)=0$. Since $(M,J)$ is of
K\"ahler type, the Hodge decomposition then tells us that $H^{0,1}(M)=0$,
and it therefore follows \cite{calabix2,ls} that every holomorphic vector field
$\Xi$ on $M$ can be written as $\nabla^{1,0}f$ for some smooth function $f=f_\Xi$, called a
holomorphy potential. This allows us to re-express the Futaki invariant as
\begin{equation}
\label{foot} 
{\mathfrak F} ( \Xi  , \Omega): = 
\left[{\mathfrak F} ( \Omega)\right]( \Xi  ) 
= - \int_M (s-\bar{s})f_\Xi \, d\mu
\end{equation}
where $\bar{s}$ denotes the average value of the scalar curvature, which 
can be computed by the topological formula
$$
\bar{s}= 4\pi m\frac{c_1\cdot \Omega^{m-1}}{\Omega^m}~.
$$
Of course,  the negative sign appearing 
in (\ref{foot}) is strictly a matter of convention, and is used  here primarily
to ensure consistency with \cite{ls}. Also note that the $\bar{s}$ term  in (\ref{foot})
could be 
dropped  if one required that the holonomy potential 
 $f_\Xi$  be normalized to have
  integral zero; however, we will find it useful to avoid  systematically   imposing 
  such  a  normalization. 
  
  Let ${\mathbf H}$ now denote the identity component
of the automorphism group of $(M,J)$, so that $\mathfrak h$ is 
its Lie algebra. Because the  assumption that $b_1(M)=0$ implies that 
 ${\mathbf H}$ is a linear algebraic  group \cite{fujikiaut},
we can define its unipotent radical ${\mathbf R}_{\mathbf u}$ to consist of 
the unipotent elements of its maximal solvable normal subgroup. 
If ${\mathbf G}\subset {\mathbf H}$ is a maximal compact subgroup, 
and if ${\mathbf G}_\CC\subset {\mathbf H}$ is its complexification, 
then ${\mathbf G}_\CC$  projects isomorphically onto  the quotient
group ${\mathbf H} /{\mathbf R}_{\mathbf u}$. The Chevalley decomposition \cite{chevalley} moreover expresses
  ${\mathbf H}$  as a semi-direct product
$${\mathbf H} = {\mathbf G}_\CC \ltimes {\mathbf R}_{\mathbf u}$$
 and we  have a corresponding split short exact sequence 
 $$
 0\to {\mathfrak r}_{\mathfrak u}\to {\mathfrak h} \to {\mathfrak g}_\CC\to 0
 $$
of Lie algebras.

 In their pioneering work on extremal K\"ahler vector fields \cite{fuma1}, Futaki and Mabuchi  
 next restricted the 
 Futaki invariant  ${\mathfrak F}$ to ${\mathfrak g}_\CC\subset {\mathfrak h}$.
 However, under mild hypotheses, this is not actually necessary: 
 
 \begin{prop}\label{sauce}
 Let $(M^{2m},J)$ be a compact complex $m$-manifold of K\"ahler type 
for which  $h^{1,0}=h^{2,0}=0$. Then the Futaki invariant
 ${\mathfrak F}(\Omega)\in {\mathfrak h}^*$ automatically 
 annihilates the Lie algebra ${\mathfrak r}_{\mathfrak u}$ of the unipotent radical,
 and so belongs to $ {\mathfrak g}_\CC^*$. Moreover, this element is automatically 
 {\em real},
 and so belongs to  $ {\mathfrak g}^*$.
 \end{prop}
 As we show in  Appendix \ref{conceptual}, this is actually a  straightforward consequence
 of   a  theorem of Nakagawa \cite{nakagawa1}. 
 
 Because the Futaki invariant is invariant under biholomorphisms, 
 it is unchanged by the action of ${\mathbf H}$
 on $\mathfrak h$. It follows that 
 ${\mathfrak F}(\Omega )$ must vanish when restricted
to the derived subalgebra $[{\mathfrak h}, {\mathfrak h}]$.
Thus,  ${\mathfrak F}(\Omega ): {\mathfrak h}\to \CC$ is actually 
a  {Lie-algebra  character}. 
In particular,  ${\mathfrak F}(\Omega )$ annihilates the derived subalgebra  
$[{\mathfrak g}, {\mathfrak g}]$ 
of the maximal compact. Since the compactness of  ${\mathbf G}$ implies that it
is  a reductive
Lie group, ${\mathfrak g}= [{\mathfrak g}, {\mathfrak g}] \oplus {\mathfrak z}$,
where $ {\mathfrak z}$ is the center of ${\mathfrak g}$. We thus conclude that 
 ${\mathfrak F}(\Omega )\in {\mathfrak z}^*$ for any K\"ahler class $\Omega$
 whenever $M$ is as in Proposition \ref{reduce}. Since 
$\mathfrak z$ is contained in the Lie algebra of any  maximal torus 
 ${\mathbf T}\subset {\mathbf G}$, we thus deduce the following important 
fact: 

\begin{prop}
\label{centerville}
Let $(M^{2m},J)$ be a compact complex $m$-manifold of K\"ahler type 
for which  $h^{1,0}=h^{2,0}=0$. 
Let ${\mathbf T}$ be a maximal torus in $\Aut (M,J)$, and 
let $\mathfrak t$ be the Lie algebra of ${\mathbf T}$. Then, for
any K\"ahler class $\Omega$ on $M$,  the Futaki invariant 
${\mathfrak F}(\Omega )$ naturally belongs
to ${\mathfrak t}^*$. In particular, ${\mathfrak F}(\Omega )$ is completely 
determined by its restriction to $\mathfrak t$. 
\end{prop} 

Now, for a fixed ${\mathbf G}$-invariant metric $g$, we have already noticed that 
every Killing field $\xi$ on $(M,g)$ is represented by a unique 
Hamiltonian  $f_\xi$ with $\int_M f_\xi \, d\mu =0$, and that the Lie bracket
on $\mathfrak g$ is thereby transformed into the Poisson bracket on 
$(M, \omega)$: 
$$
f_{[\xi , \eta ]} = \{ f_\xi , f_\eta \}= -\omega^{-1}(df_\xi , df_\eta )~.
$$
Following Futaki and Mabuchi \cite{fuma1}, we may therefore introduce a bilinear form 
${\mathbb B}$ on the real Lie algebra ${\mathfrak g}$ by restricting the 
$L^2$ norm of $(M,g)$ to the space of these Hamiltonians:
$$
{\mathbb B}(\xi , \eta ) = \int_M f_\xi f_\eta \, d\mu_g = \frac{1}{m!}\int_M f_\xi f_\eta \, \omega^m~.
$$
 Since a straightforward version of Moser stability shows that 
 the K\"ahler forms of any two $\mathbf G$-invariant metrics in a fixed K\"ahler class are 
 $\mathbf G$-equivariantly symplectomorphic, this inner product only depends
 on $\Omega$ and  the maximal compact ${\mathbf G}< {\mathbf H}$, not on the representative
  metric
 $g$. Moreover, since any two maximal compacts are conjugate in ${\mathbf H}$,
 one can show \cite{fuma1} that the corresponding complex-bilinear form on 
 ${\mathfrak g}_\CC= {\mathfrak h}/{\mathfrak r}_{\mathfrak u}$ is actually
 independent  of the choice of maximal compact ${\mathbf G}$. 
 
 Since ${\mathbb B}$ is positive-definite, and so defines an isomorphism 
 ${\mathfrak g} \to {\mathfrak g}^*$, it also has a well-defined inverse 
 which gives a positive-definite bilinear form 
 $${\mathbb B}^{-1} : {\mathfrak g}^*\times {\mathfrak g}^*\to \RR$$
on the Lie coalgebra of our maximal compact. 
On the other hand, assuming that $(M,J)$ is as in Proposition \ref{reduce}, 
we have already seen that ${\mathfrak F}(\Omega )\in {\mathfrak g}^*$
for any K\"ahler class $\Omega$ on $M$. 
Thus, the number 
\begin{equation}
\label{norm}
\| {\mathfrak F}(\Omega )\|^2 : = {\mathbb B}^{-1}({\mathfrak F}(\Omega ) , {\mathfrak F}(\Omega ))
\end{equation}
is independent of choices, and so is an invariant of $(M,J, \Omega )$.

To see why this number has an important differential-geometric significance,
let us first suppose that $g$ is a ${\mathbf G}$-invariant K\"ahler metric with 
K\"ahler class $\Omega$, 
and let $\proj$ be  orthogonal projection in the real Hilbert space
$L^2(M,g)$ to the subspace of normalized Hamiltonians representing the Lie algebra
${\mathfrak g}$ of Killing fields on $(M,g)$. Restricting 
equation  (\ref{foot}) to ${\mathfrak g} \subset {\mathfrak h}$, one 
observes  that ${\mathfrak F}(\Omega) : {\mathfrak g} \to \RR$
is exactly given by the ${\mathbb B}$-inner-product with the Killing field
whose Hamiltonian is $- \proj(s-\bar{s})$. We thus immediately
have
$$
\int_M \left[{\proj}(s-\bar{s})\right]^2d\mu_g =  \| {\mathfrak F}(\Omega )\|^2 
$$
and, since the projection $\proj$ is norm-decreasing, it follows  that 
\begin{equation}
\label{minoree} 
\int_M (s-\bar{s})^2 d\mu_g \geq \| {\mathfrak F}(\Omega )\|^2
\end{equation}
for any ${\mathbf G}$-invariant K\"ahler metric with 
K\"ahler class $\Omega$. It is a remarkable  fact,
proved by Xiuxiong Chen \cite{xxel},  that  inequality (\ref{minoree}) 
actually holds even if $g$ is {\em not} assumed to be ${\mathbf G}$-invariant. 
 Moreover, equality holds  in (\ref{minoree}) if and only if
  $\nabla^{1,0}s$ is a holomorphic vector field, which is precisely the condition 
\cite{calabix,calabix2}
for $g$ to be an  extremal K\"ahler metric.

 The bilinear form ${\mathbb B}$ on $\mathfrak g$ is bi-invariant. In particular, 
 the center $\mathfrak z$ of $\mathfrak g$ is ${\mathbb B}$-orthogonal to 
 the semi-simple
 factor $[{\mathfrak g}, {\mathfrak g}]$ of ${\mathfrak g}$. Thus, 
 a computation of $\|\mathfrak F (\Omega)\|^2$ does not require 
 a complete knowledge of the bilinear form $\mathbb B$; only a knowledge of
 its  restriction  to 
 $\mathfrak z$ is required. This observation allows us to prove  the following:
 
 \begin{cor} \label{pump} 
 Let $(M,J)$ be as in Proposition \ref{centerville},  let $\mathbf T$ be a 
 maximal torus in the complex automorphism group of $(M,J)$,
 and let $\mathfrak t$ denote the Lie algebra of 
$\mathbf T$.
If  $g$ is 
  any ${\mathbf T}$-invariant K\"ahler metric with K\"ahler class $\Omega$, and if 
$${\mathbb B}_{\mathbf T}: {\mathfrak t}\times {\mathfrak t}\to \RR$$ is 
the  $g$-induced $L^2$-norm restricted to normalized Hamiltonians, then
$$\| {\mathfrak F}(\Omega )\|^2 = {\mathbb B}_{\mathbf T}^{-1}\Big(   {\mathfrak F}(\Omega ) \, ,  \, {\mathfrak F}(\Omega )\Big)$$
where ${\mathbb B}_{\mathbf T}^{-1}$ denotes the inner product on ${\mathfrak t}^*$ induced by 
${\mathbb B}_{\mathbf T}$.
 \end{cor}
 \begin{proof}
 Let ${\mathbf G}$ be a maximal compact subgroup of ${\mathbf H}$ containing $\mathbf T$. 
 Then, by  Proposition \ref{centerville}, 
 the  assertion  certainly holds for any ${\mathbf G}$-invariant K\"ahler metric $\tilde{g}$ in $\Omega$.
 However, by averaging, any ${\mathbf T}$-invariant K\"ahler metric with K\"ahler class $\Omega$
 can be joined to $\tilde{g}$ by a path of such metrics, and is therefore ${\mathbf T}$-equivariantly
 symplectomorphic to $\tilde{g}$ by Moser stability. The claim therefore follows, since 
  ${\mathfrak F} (\Omega)\in {\mathfrak t}^*$ is completely
 determined by $(M,J,\Omega)$, while 
 ${\mathbb B}_{\mathbf T}$ is completely determined by the symplectic form 
 and normalized  Hamiltonians representing elements of $\mathfrak t$. 
 \end{proof}

\section{Toric Manifolds} 

We now specialize our discussion to the toric case. For  clarity, our presentation will be 
  self-contained, and will 
  include  idiosyncratic proofs of various  standard facts about  toric geometry. 
 For  more orthodox  expositions of some of these  fundamentals, 
 the reader might do well  to consult  
\cite{fultor} and \cite{guiltor}.

We  define  a {\em toric manifold} to be 
  a  (connected) compact complex $m$-manifold  $(M^{2m}, J)$ of K\"ahler type
which has  non-zero Euler characteristic and which is 
  equipped a group of automorphisms generated by  $m$ commuting, periodic,  
  $J$-preserving real vector fields 
  which are linearly
  independent in the space of vector fields on $M$. Thus,   the 
  relevant group  of automorphisms $\mathbf T$ is required to be the image of 
  the $m$-torus under 
  some Lie group homomorphism 
  $T^m\to \Aut (M, J)$ which induces an injection of Lie algebras.
  Notice that our definition implies that 
  there must be a fixed point $p\in M$ of  this $T^m$-action. Indeed, the
  fixed point set of any circle action on a smooth compact manifold is  \cite{kobfixed} a
  disjoint union of smooth compact manifolds with total Euler characteristic
  equal to the  Euler characteristic of the ambient space;  by induction on the number of   circle factors, it follows  that the 
  fixed-point set of any torus action on $M$ therefore has total Euler characteristic 
  $\chi (M)\neq 0$,
  and so, in particular, cannot be empty. 
  
  In light of this, let $p\in M$ be a fixed point of the given $T^m$-action on a toric manifold
   $(M^{2m},J)$, and, by averaging, also choose
  a K\"ahler metric  $g$ on $M$ which is  $T^m$-invariant.  Then $T^m$ acts on $T_pM\cong \CC^m$ in a manner preserving both $g$ and $J$, 
 giving us  a unitary representation $T^m\to {\mathbf U}(m)$. Since 
 the action of $T^m$ on $T_pM$
 completely determines the action on $M$ via the exponential map $T_pM\to M$ of $g$, and 
 since, by hypothesis,  
 the Lie algebra of $T^m$ injects into the vector fields on $M$, it follows that 
 the above unitary representation gives rise to a faithful representation of 
 ${\mathbf T}<\Aut (M,J)$. However,  ${\mathbf U}(m)$ has rank $m$, so 
 the image of  $T^m\to {\mathbf U}(m)$  must be  a maximal torus
 in ${\mathbf U}(m)$; thus, after a change of basis of $\CC^m$, 
${\mathbf T}$  may be identified with the standard
 maximal torus ${\mathbf U}(1)\times \cdots \times {\mathbf U}(1)
 \subset  {\mathbf U}(m)$  consisting of  diagonal matrices.
 In particular, $\mathbf T< \Aut (M,J)$ is intrinsically an $m$-torus, and 
 has many free orbits.
 Since the 
 origin in $\CC^m$ is the only fixed point of  the diagonal torus in ${\mathbf U}(m)$,  
 it also follows that $p$ must be an isolated fixed point of ${\mathbf T}$. 
But since the same argument applies equally well to any other fixed point, this shows that 
the fixed-point set $M_\mathbf T$ of $\mathbf T$ is discrete, and therefore finite.
In particular, $\chi (M)$  must equal the 
  cardinality of  $M_\mathbf T$, so 
  the Euler characteristic of $M$ is necessarily positive.

 The above arguments  in particular show  that  the toric condition can be  
 reformulated as follows: 
  a toric $m$-manifold is a compact complex $m$-manifold $(M,J)$ of K\"ahler type, 
  together  with an $m$-torus ${\mathbf T}\subset \Aut (M,J)$ that 
  has both a free orbit ${\mathcal Q}$ and a fixed point $p$.  To check the equivalence, note
  that 
  this reformulation implies that the Euler characteristic $\chi (M)$ is positive, because
  the  fixed-point set $M_{\mathbf T}$
  is necessarily   finite,  and  by hypothesis is also non-empty.

Now let $(M, J , {\mathbf T})$ be   a toric $m$-manifold, and let 
${\mathfrak j}: {\mathcal Q}\hookrightarrow M$ be the inclusion of a
free $\mathbf T$-orbit. Since $\mathbf T$ also has a fixed-point $p$, and since 
  any two $\mathbf T$-orbits are homotopic, it follows that 
${\mathfrak j}$ is  
homotopic to a constant map. Consequently,  the induced homomorphism  
${\mathfrak j}^\ast :
H^k (M) \to H^k({\mathcal Q})$ must be the zero map  in all dimensions
 $k > 0$.  However, 
the restriction of the   K\"ahler form $\omega = g (J\cdot , \cdot )$
to ${\mathcal Q} \approx {\mathbf T}$  is  an  invariant $2$-form on ${\mathbf T}\approx T^m$. 
Since every deRham class on $T^m$ contains a unique 
invariant form, and since ${\mathfrak j}^*[\omega ] = 0 \in H^2 (T^m, \RR)$, 
it follows that ${\mathfrak j}^*\omega$ must vanish identically. Thus 
${\mathcal Q}$ is a Lagrangian submanifold, which is to say that  $T{\mathcal Q}$ is everywhere orthogonal to 
$J(T{\mathcal Q})$. In particular, if $\xi_1 , \ldots , \xi_m$ are the generators 
of the ${\mathbf T}$-action, the corresponding holomorphic   vector fields 
  $\Xi_j = -(J\xi_j +i\xi_j)/2$
 span $T^{1,0}M$ in a neighborhood of ${\mathcal Q}$. 
Integrating the flows of the commuting vector fields $\xi_j$ and $J\xi_j$,
we thus obtain 
a holomorphic action of the complexified torus $(\CC^\times)^m$  which has
 both a fixed point and an open orbit $\mathcal U$.

 In particular, $(M,J)$ carries 
 $m$ holomorphic vector fields $\Xi_1, \ldots , \Xi_m$ which vanish at $p$, 
 but which nonetheless span $T^{1,0}(M)$ at a generic point. 
 It follows that $M$ cannot carry a non-trivial holomorphic $k$-form
 $\alpha \in H^{k,0}(M)$ for any $k > 0$, since,  for any choice 
 of $j_1, \ldots , j_k$,  the ``component'' 
 functions $\alpha (\Xi_{j_1}, \ldots , \Xi_{j_k})$
 would be holomorphic, and hence  constant,  and yet would have to vanish at the fixed point $p$. In particular, we may invoke   Kodaira's observation  \cite{kodrr} that any 
 K\"ahler manifold with  $H^{2,0}=0$ admits Hodge metrics, and so is projective. This
 gives us the following result:

 \begin{lem}\label{van} 
 Any toric manifold $M$ is projective algebraic, and satisfies \linebreak 
  $H^{k,0}(M)=0$ for all $k> 0$. 
 \end{lem}

In particular, the identity component 
 ${\mathbf H} = \Aut^0(M,J)$
  of the automorphism group of our toric $m$-manifold 
  is linear algebraic. Let ${\mathbf T} < {\mathbf H}$ be the $m$-torus associated
  with the toric structure of $(M,J)$. 
 Using the Chevalley decomposition, we 
  can then choose  a  
 maximal compact subgroup 
 $\mathbf G < {\mathbf H}$ containing $\mathbf T$. Also choose   a $\mathbf G$-invariant 
 K\"ahler metric $g$ on $M$  and 
    a  fixed point $p$  of   ${\mathbf T}$.
  We will now study the centralizer $Z({\mathbf T})<\mathbf G$, consisting of elements 
   of $\mathbf G$ that
   commute with all elements of ${\mathbf T}$. Observe that 
   $${\zap a}\in Z({\mathbf T}), {\zap b} \in {\mathbf T} ~\Longrightarrow 
   {\zap b}({\zap a}(p)) = {\zap a}({\zap b}(p))= {\zap a}(p),$$
   so that $Z({\mathbf T})$ acts by permutation  on the finite set $M_{\mathbf T}$ of 
   fixed points. 
   In particular, the identity component $Z^0({\mathbf T})$ of $Z({\mathbf T})$
   must send $p$ to itself. Once more invoking the exponential map of $g$,
   we thus obtain a faithful unitary representation of  $Z^0({\mathbf T})$
   by considering its induced action on $T_pM\cong \CC^m$. However, the 
   image of $Z^0({\mathbf T})$ in ${\mathbf U}(m)$ must then be a  subgroup of
   the centralizer of the diagonal torus ${\mathbf U}(1)\times \cdots \times {\mathbf U}(1)$
   in  ${\mathbf U}(m)$. But since the latter  centralizer is just the diagonal torus itself,
   we conclude  that 
 $Z^0({\mathbf T})= {\mathbf T}$. It follows that  ${\mathbf T}$ is a maximal torus in 
 $\mathbf G$,   and hence  also  in $\mathbf H = {\mathbf G}_\CC \ltimes {\mathbf R}_{\mathbf u}$:

\begin{lem} \label{max} 
Let $(M^{2m}, J)$ be a toric manifold, and let    ${\mathbf T}< \Aut (M,J)$
be the associated $m$-torus. 
 Then ${\mathbf T}$ is a {maximal} torus in $\Aut (M,J)$. 
\end{lem}

Combining this result with Lemma \ref{van} and Proposition \ref{centerville}, we 
can thus  generalize \cite[Theorem 1.9]{nakagawa2} to irrational 
K\"ahler classes: 
 
 \begin{prop}
\label{toroidal}
Let $(M^{2m},J)$ be a toric manifold,  let ${\mathbf T}$
be the given  $m$-torus in its automorphism group, and 
let $\mathfrak t$ be the Lie algebra of ${\mathbf T}$. Then, for
any K\"ahler class $\Omega$ on $M$,  the Futaki invariant 
${\mathfrak F}(\Omega )$ naturally belongs
to ${\mathfrak t}^*$. In particular, ${\mathfrak F}(\Omega )$ is completely 
determined by its restriction to $\mathfrak t$. 
\end{prop}

However, we will not simply need to know
 where ${\mathfrak F}(\Omega )$  lives;  our goal will require us   to calculate 
 its norm with respect to the relevant bilinear form. 
Fortunately,  Lemma \ref{max} and   Corollary \ref{pump} together 
imply   the following result: 
 
 \begin{prop} \label{bump}
Let $(M^{2m},J)$ be a toric manifold,  let ${\mathbf T}$
be the given  $m$-torus in its automorphism group, and 
let $\mathfrak t$ be the Lie algebra of ${\mathbf T}$. 
If  $g$ is 
  any ${\mathbf T}$-invariant K\"ahler metric with K\"ahler class $\Omega$, and if 
$${\mathbb B}_{\mathbf T}: {\mathfrak t}\times {\mathfrak t}\to \RR$$ is 
the  $g$-induced $L^2$-norm restricted to normalized Hamiltonians, then
$$\| {\mathfrak F}(\Omega )\|^2 = {\mathbb B}_{\mathbf T}^{-1}\Big(   {\mathfrak F}(\Omega ) \, ,  \, {\mathfrak F}(\Omega )\Big)$$
where ${\mathbb B}_{\mathbf T}^{-1}$ denotes the inner product on ${\mathfrak t}^*$ induced by 
${\mathbb B}_{\mathbf T}$.
 \end{prop}

  Of course, Lemma  \ref{van}  has many other interesting applications. For example, 
  by Hodge symmetry, it implies 
 the Todd genus is given by 
    $$\chi (M, {\mathcal O})=\sum_k (-1)^kh^{0,k}(M)=1$$ for any toric manifold $M$. 
    Since the same argument could also be applied to any finite covering of $M$,
    whereas $\chi (M, \mathcal O )$ is multiplicative under coverings, one immediately
    sees that $M$ cannot have non-trivial finite covering spaces. In particular, 
     this  implies  that  $H_1 (M, \ZZ)=0$. 

 However, one can easily do much better.
 Choose  a $\mathbf T$-invariant
 K\"ahler metric $g$ with K\"ahler form $\omega$.
 Because $b_1(M) = 0$ by Lemma  \ref{van}, 
 the symplectic vector fields $\xi_1 , \ldots , \xi_m$
must then have Hamiltonians, so that $\xi_j=J\nabla f_j$ for suitable functions
 $f_1, \ldots , f_m$. Let $a_1, \ldots , a_m$ be real numbers which are
 linearly independent over $\QQ$, and let $f= \sum_j a_j f_j$. The corresponding 
 symplectic vector field $\xi=\sum_j a_j \xi_j$ is thus a Killing field for $g$, and 
 its flow is dense in the torus ${\mathbf T}<\Aut (M,J)$. Consequently, 
 $\xi$ vanishes only at the fixed points of ${\mathbf T}$. Since $\xi$ is Killing, with
 only isolated zeroes, it then follows that  $\nabla\xi$ 
 is  non-degenerate at each fixed point $p$ of $\mathbf T$, in the sense
 that it defines an isomorphism $T_p\to T_p$. 
 Since $\nabla_a\nabla_bf=\omega_{bc}\nabla_a\xi^c$, 
this implies  that the Hessian of $f$
 is non-degenerate at each zero of $df$; that is, $f$ is a Morse function 
 on $M$. However, since $\xi$ is the real part of a holomorphic vector field,
 $\bar{\partial}\partial^\# f=0$, and this is equivalent to saying that  the Riemannian Hessian 
  $\nabla \nabla f$  is everywhere $J$-invariant. Since the Riemannian Hessian coincides
  with the na\"{\i}ve  Hessian at a critical point, this shows that  every critical point of 
 $f$ must have even index. It follows  \cite{milmo} that $M$ is homotopy
 equivalent to a $CW$ complex consisting entirely of even-dimensional cells. 
In particular, we obtain the following:

 \begin{lem} \label{simple} 
 Any toric manifold  is simply connected, and has trivial homology in all  odd
 dimensions.  \end{lem}

Finally, notice that Lemma \ref{van} implies  that the canonical line bundle
$K=\Lambda^{m,0}$ of a toric $m$-manifold has no non-trivial holomorphic 
sections. However, essentially
the same argument also shows that positive powers $K^\ell$ cannot have
non-trivial holomorphic  sections either, since the pairing of such a section with 
$(\Xi_1 \wedge \cdots  \wedge \Xi_m)^{\otimes \ell}$ would again result in a  constant
function which would have to vanish at $p$. Thus, all the plurigenera 
$p_\ell = h^0({\mathcal O}(K^{\ell}))$ of any toric manifold must vanish. In other words:

 \begin{lem} \label{helsing} 
 Any  toric manifold  has Kodaira dimension 
 $-\infty$.
 \end{lem}

\section{The  Virtual Action}

As  previously discussed  in connection with  (\ref{minoree}), a theorem of
Chen \cite{xxel} says that any K\"ahler metric $g$ on a compact 
complex manifold $M$ satisfies 
\begin{equation}
\label{shine}
\int_M (s-\bar{s})^2 d\mu_g \geq \| {\mathfrak F}(\Omega )\|^2~,
\end{equation}
where $\Omega =[\omega ]$ is the K\"ahler class of $g$; moreover, 
equality  holds iff $g$ is an extremal K\"ahler metric. On the other hand, 
\begin{equation}
\label{python} 
\int_M s^2 d\mu_g=\int_M (s-\bar{s})^2 d\mu_g+ \int_M \bar{s}^2 d\mu_g
\end{equation}
as may be seen 
by applying  the Pythagorean theorem to $L^2$-norms. Since $s$ is the trace of 
the Ricci tensor with respect to the metric, and because the Ricci form 
is essentially the curvature of the canonical line bundle, we also know that 
\begin{equation}
\label{ints} 
\int_M s~d\mu = \frac{ ~4\pi c_1\cdot\Omega^{m-1}}{(m-1)!}
\end{equation}
in complex dimension $m$; meanwhile, the volume of an $m$-dimensional 
K\"ahler $m$-manifold is just given by 
$$
\int_M d\mu = \frac{~\Omega^m}{m!}~.
$$
Hence
$$
\int_M  \bar{s}^2d\mu = \frac{\left(\int_Ms~d\mu\right)^2}{\int_M  d\mu}= 
\frac{16\pi^2 m}{(m-1)!}\frac{(c_1\cdot \Omega^{m-1})^2}{\Omega^m}
$$
and \eqref{shine} thus implies that 
\begin{equation}
\label{higherup} 
\int_M s^2 d\mu_g\geq
\frac{16\pi^2 m}{(m-1)!}\frac{(c_1\cdot \Omega^{m-1})^2}{\Omega^m}
+ \|{\mathfrak F}(\Omega) \|^2
\end{equation}
with equality iff  $g$ is an extremal K\"ahler  metric on $(M^{2m},J)$. 

Now specializing  to the case of complex dimension $m=2$,  we have 
$$
\int_M s^2 d\mu_g\geq
32\pi^2 \frac{(c_1\cdot \Omega )^2}{\Omega^2}+ \|{\mathfrak F}(\Omega) \|^2
$$
for any K\"ahler metric $g$ with K\"ahler class $[\omega ] = \Omega$ on a compact 
complex 
surface $(M^4, J)$. 
In other words, if we define a function on the 
K\"ahler cone  by 
$$
{\mathcal A} (\Omega ):=
\frac{(c_1\cdot \Omega )^2}{\Omega^2}+ \frac{1}{32\pi^2}\|{\mathfrak F}(\Omega) \|^2~,
$$
then 
\begin{equation}
\label{better}
\frac{1}{32\pi^2} \int_M s_g^2~d\mu_g \geq {\mathcal A}(\Omega)
\end{equation}
for any K\"ahler metric $g$ with K\"ahler class $\Omega$, with equality
iff $g$ is an extremal K\"ahler metric. The function ${\mathcal A}(\Omega)$
will be called the {\em virtual action}. Our normalization has been chosen so that 
${\mathcal A}(\Omega)\geq c_1^2 (M)$, with equality iff the Futaki invariant vanishes
and $\Omega$   is a multiple of   $c_1$. (Incidentally,  the latter  occurs
 iff  $\Omega$ is the K\"ahler class of a 
K\"ahler-Einstein metric    on $(M^4, J)$ \cite{aubin,sunspot,tian,yauma}.) The fact that the virtual action 
${\mathcal A}(\Omega)$ is homogeneous of
degree $0$ in $\Omega$ corresponds to the fact that the Calabi energy 
${\mathcal C}(g)$ is scale-invariant in real dimension four.

In complex dimension $m=2$,
 one important reason for studying the Calabi energy ${\mathcal C}$ is
the manner in which (\ref{weyl}) relates it to the {\em Weyl functional} 
$${\mathcal W}(g)= \int_{M} |W|^{2}_gd\mu_{g}$$
where the Weyl curvature $W$ is the
conformally invariant piece of the curvature tensor.
It is easy to check that $\mathcal W$ is also conformally invariant, and
may therefore be considered as a functional on the space of conformal classes
of Riemannian metrics.  Critical points of the Weyl
 functional are characterized \cite{bach,bes} by 
the vanishing of the {\em Bach tensor}  
$$B_{ab}:=(\nabla^c\nabla^d+\frac{1}{2}r^{cd})
W_{acbd}~$$  
and so are said to be  {\em Bach-flat};   obviously, this is  a
conformally invariant condition. 
The  Bianchi identities immediately imply that any Einstein metric
on a $4$-manifold is Bach-flat, and it therefore follows that any 
conformally Einstein metric is Bach-flat, too.  
The  converse, however,   is  false;
for example, self-dual and anti-self-dual metrics are also 
Bach-flat, and such metrics   exist on many compact $4$-manifolds  \cite{mcp2,lebicm,tasd}
 that do not admit 
Einstein metrics. 

When the Weyl functional ${\mathcal W}$ is restricted to the space 
of K\"ahler metrics,
equation  (\ref{weyl}) shows that it becomes equivalent to the Calabi energy ${\mathcal C}$. 
Nonetheless, the following result \cite{chenlebweb} may come as something of a surprise: 

\begin{prop} 
\label{bfk}
Let  $g$ be a K\"ahler metric  on a compact complex surface. 
Then $g$ is Bach-flat if and only if  
\begin{itemize}
\item $g$ is an extremal K\"ahler metric, and
\item  its 
  K\"ahler class $\Omega$
is a  critical point of the virtual action
${\mathcal A}$. 
\end{itemize}
\end{prop}
This gives rise to  a remarkable method  of constructing Einstein metrics, courtesy of a
 beautiful result of Derdzi{\'n}ski \cite[Proposition 4]{derd}: 
 
 \begin{prop}
 If 
 the scalar curvature $s$ of a Bach-flat K\"ahler metric $g$ on a complex surface $(M^4,J)$ 
 is not identically zero,
then the conformally related metric $h=s^{-2}g$ is Einstein on the open set $s\neq 0$ where it is defined.
\end{prop}

\section{Toric Surfaces}
\label{focus} 

We will  now prove Theorem \ref{behold} by   computing 
 the virtual action ${\mathcal A} (\Omega )$ for
any K\"ahler class on a toric surface. An important intermediate step in this process
involves an  explicit computation of the 
Futaki invariant ${\mathfrak F}(\Omega )$. Up to a  universal constant, our answer agrees
with that of  various other authors 
  \cite{dontor,fuma1,mabary,shelukhin}, but determining   the correct constant is crucial 
 for our  purposes. For this reason, our first proof 
 will be  based on the author's formula \cite{mcp2} for the scalar curvature of 
 a K\"ahler surface with isometric $S^1$ action.

By  a {\em toric surface}, we mean 
a toric manifold $(M,J, {\mathbf T})$ of complex dimension two. 
This is equivalent\footnote{In one direction,
this equivalence follows  because any simply connected 
compact complex surface is of K\"ahler type \cite{nick,siu}
and has positive Euler characteristic. On the other hand, the converse  follows  from 
Lemma \ref{simple}. } to saying that 
$(M^4, J)$ is  a   simply connected compact complex surface 
equipped with a $2$-torus ${\mathbf T} < \Aut (M,J)$.
By Castelnuovo's criterion \cite{bpv,gh},
Lemma \ref{van} and Lemma \ref{helsing}, any toric surface 
 $(M,J)$   can be obtained from either $\CP_2$ or
 a Hirzebruch surface by blowing up points. Indeed, since the holomorphic vector fields 
 generating the torus action on $M$ automatically descend to the minimal model,
 the toric structure of $(M,J)$ can be obtained from a toric structure on  $\CP_2$ or
 a Hirzebruch surface
 by iteratively blowing up fixed points of the torus action. For more direct proofs, using the toric
 machinery of 
  fans or moment polytopes,  see \cite{fultor,guiltor}.

 Let   $(M^4, J, {\mathbf T}, \Omega)$ now be a toric surface with fixed K\"ahler class. 
By averaging, we can then find a $\mathbf T$-invariant K\"ahler metric $g$ on $(M,J)$ with 
K\"ahler form $\omega \in \Omega$. Choose an isomorphism 
${\mathbf T}\cong \RR^2/\ZZ^2$,
and denote the corresponding generating vector fields of period $1$
by $\xi_1$ and $\xi_2$. 
Since $b_1(M)=0$, there are Hamiltonian functions $x_1$ and $x_2$ on 
$M$ with $\xi_j= J\grad x_j$, $j=1,2$. This makes $(M, \omega)$ into
a Hamiltonian $T^2$-space in the sense of \cite{guiltor}. In particular, the image of $M$
under  $\vec{x}=(x_1,x_2)$ is  \cite{atcvx,gscvx}
 a convex  polygon $P\subset \RR^2$ whose area is
exactly  the volume of 
$(M,g)$.  
The map $\vec{x}: M\to \RR^2$ is called the {\em moment map}, and 
 its image $P=\vec{x} (M)$ will be called the 
{\em moment polygon}. Of course, since we have not
insisted that 
the Hamiltonians $x_k$  have integral zero,  our moment map is  only  determined 
up to  translations of
$\RR^2$. Modulo this ambiguity, however, the moment polygon is
 uniquely determined by $(M,\omega, {\mathbf T})$, together with the 
 chosen basis $(\xi_1, \xi_2)$ for the Lie algebra $\mathfrak t$ of of $\mathbf T$. Moreover, 
since a straightforward  Moser-stability argument shows that any two $\mathbf T$-invariant 
K\"ahler forms in $\Omega$ are $\mathbf T$-equivariantly symplectomorphic, the moment polygon
  really only  depends on $(M,J, \Omega, (\xi_1, \xi_2))$.
However, outer automorphisms of $\mathbf T$ can be used to alter $(\xi_1, \xi_2)$  
by an $\mathbf{SL}(2,\ZZ)$ transformation, and this in turn changes the moment polygon by 
an $\mathbf{SL}(2,\ZZ)$ transformation of $\RR^2$. Moreover, since the vertices of 
$P$ correspond to the fixed points of $\mathbf T$, and because the action of 
$\mathbf T$ on the tangent space  of any fixed point can be identified 
with that of the diagonal torus ${\mathbf U}(1)\times {\mathbf U}(1)\subset {\mathbf U}(2)$, 
 a neighborhood of  any   corner of $P$ 
can be transformed into   a neighborhood of the origin in the positive quadrant of  $\RR^2$
by  an element of $\mathbf{SL}(2,\ZZ)$  and a translation \cite{delzant}. 
Polygons with the latter property are said to be  {\em Delzant}, and 
 any Delzant polygon arises from a uniquely determined toric surface,
 equipped with a uniquely
determined K\"ahler class \cite{guiltor}. 

We now introduce 
a measure $d\lambda$ on the boundary $\partial P$ of our moment polygon. To
do this, first notice that each edge
of $P$ is the image of a rational curve ${C}_{\imath}\cong \CP_1$ in $(M,J)$ which is fixed by an $S^1$ 
subgroup of $T^2$, and hence by a $\CC^\times$ subgroup of the complexified 
torus  $\CC^\times \times \CC^\times$. We then define the measure $d\lambda$ along the 
edge $\ell_\imath= \vec{x}(C_\imath)$ to be the push-forward, via $\vec{x}$, of the smooth area 
measure on $C_\imath$ given by the restriction  of the K\"ahler form $\omega$. 
Since a rational linear combination of the $x_k$ is a Hamiltonian for
rotation of $C_\imath$ about  two fixed points, 
$d\lambda$  is a constant times $1$-dimensional Lebesgue measure on the
line segment $\ell_\imath$, with total length 
$$\int_{\ell_\imath} d\lambda = \int_{C_\imath}\omega  = : {\zap A}_\imath$$
equal to the area of corresponding holomorphic curve in $M$. Here the index $\imath$ 
is understood to run  over the edges of $\partial P$. 

When an edge is parallel to either axis, 
$d\lambda$  just becomes standard Euclidean length measure. 
More generally, on an arbitrary edge, it must coincide with the pull-back of Euclidean length 
via any  $\mathbf{SL}(2, \ZZ )$ 
transformation which sends the edge to a segment parallel to an axis. 
Because $P$ is a Delzant polygon,
this contains
enough information to completely 
 determine  $d\lambda$,   and  leads to 
  a consistent definition of the measure because the stabilizer 
$$
\left\{\left. 
\pm \left(\begin{array}{cc} 1 & k \\0 &  1\end{array}\right)~\right|~ k\in \ZZ
\right\}
$$
of the $x_1$-axis in $\mathbf{SL}(2,\ZZ )$  preserves Euclidean length 
 on this axis. However, the Euclidean algorithm of elementary number theory implies that 
 every pair $(p,q)$ of relatively prime non-zero integers belongs to the $\mathbf{SL}(2, \ZZ )$-orbit of 
 $(1,0)$. One can therefore compute edge-lengths with respect to 
 $d\lambda$ by means of  the following recipe: 
 Given an edge  of $P$ which is not parallel to either axis, its slope 
 $m$ is a non-zero rational number, and so can be expressed 
 in {lowest terms} as  $m=q/p$, 
where $p$ and $q$ are relatively prime non-zero integers.
The displacement vector $\vec{v}$ representing the 
difference between the two endpoints of the edge can thus be written as 
$\vec{v}= (up,uq)$ for some $u\in \RR -\{0\}$. The  
length of the edge with respect to $d\lambda$ then equals $|u|$. 

We can now  associate two different barycenters  with our moment polygon. 
First, there is the barycenter $\bar{\vec{x}}= (\bar{x}_1, \bar{x}_2)$ of the interior
of $P$, as defined by 
$$\bar{x}_k = \fint_P x_k ~d{\zap a} = \frac{\int_P x_k d{\zap a}}{\int_P  d{\zap a}}$$
where $d{\zap a}$ is standard $2$-dimensional Lebesgue measure in $\RR^2$. 
Second, there is the barycenter $\langle\vec{x}\rangle= (\langle {x}_1\rangle, \langle {x}_2\rangle)$ of the perimeter $\partial P$, defined  by 
$$
\langle x_k\rangle = \fint_{\partial P} x_k ~d\lambda=
 \frac{\int_{\partial P} x_k d\lambda}{\int_{\partial P}  d\lambda}
$$
These two barycenters certainly need not coincide in general. It is   therefore natural
 to consider the 
 displacement vector 
$$
\vec{\mathfrak D} = \langle\vec{x}\rangle - \bar{\vec{x}}
$$
that measures their separation. Notice that $\vec{\mathfrak D}$ is translation invariant ---
it is unchanged if we alter the Hamiltonians $(x_1,x_2)$ by adding  constants. 

Next, we introduce the moment-of-inertia matrix $\Pi$ of $P$, which encodes
the moment of inertia of the polygon about an arbitrary axis in $\RR^2$ passing 
through its barycenter
$\bar{\vec{x}}$. Thus $\Pi$ is the positive-definite symmetric $2\times 2$ matrix with 
entries given by  
$$
{\Pi}_{jk} =\int_P (x_j-\bar{x}_j)(x_k-\bar{x}_k) d{\zap a}
$$
where 
$d{\zap a}$ once again denotes the usual Euclidean area form on the interior of $P$, and 
exactly  equals the push-forward of
the metric volume measure on $M$. 
For our purposes, it is important to notice that $\Pi$ is always an invertible matrix.

Finally, let $|\partial P|= \int_{\partial P} d\lambda = \sum_\imath {\zap A}_\imath$ denote
the perimeter of the moment polygon with respect to the measure $d\lambda$ introduced above, and let
 $|P| = \int_Pd{\zap a}$ denote the area of its interior in the usual sense.
With these notational conventions,  we are now ready to 
state the main result of this section:

\begin{thm} \label{chiaro} 
If $(M, J, \Omega)$ is any toric surface with fixed K\"ahler class,  then 
\begin{equation}
\label{sounder}
{\mathcal A}(\Omega) =\frac{|\partial P|^2}{2}\Big( \frac{1}{|P|~} + \vec{\mathfrak D}\cdot  {\Pi}^{-1} \vec{\mathfrak D}
\Big)
\end{equation}
where $P$ is the moment polygon determined by the given $T^2$-action. 
\end{thm}
 
The proof of Theorem \ref{chiaro} crucially depends on a computation of the
Futaki invariant, which, we recall, 
 is a character on the Lie algebra of holomorphic vector fields. 
Let us therefore consider the holomorphic 
vector fields  $\Xi_k= \nabla^{1,0}x_k$ whose holomorphy
potentials are the Hamiltonians of the periodic Killing fields 
$\xi_k$. These are explicitly given
by 
$$
\Xi_k = -\frac{1}{2} \left( J\xi_k + i \xi_k\right). 
$$

\begin{prop} \label{stage} 
Suppose that $(M,J,\Omega)$ is a toric surface with fixed K\"ahler class, and 
let $\Xi_k$ be the    generators of  the 
associated complex torus action,  normalized as above. Let 
$${\mathfrak F}_k:= {\mathfrak F}(\Xi_k,\Omega )$$
be the corresponding components of the Futaki invariant of $(M,J,\Omega)$. 
Then the vector $\vec{\mathfrak F}= ({\mathfrak F}_1, {\mathfrak F}_2)$
is explicitly given by 
$$\vec{\mathfrak F} = -4\pi \, |\partial P| \, \vec{\mathfrak D}$$
where $|\partial P|$  again denotes the weighted perimeter of the 
moment polygon $P$, and  $\vec{\mathfrak D}$ is  again the vector joining
the  barycenters of the interior and weighted boundary of  $P$. 
\end{prop}

\begin{proof} More explicitly, the assertion is  that 
\begin{equation}
\label{norm}
{\mathfrak F}_k = -4\pi  \sum_{\imath}
\Big(\langle x_k \rangle_\imath- \bar{x}_k\Big)~{\zap A}_\imath
\end{equation}
where $\bar{x}_k$ is once again the $k^{\rm th}$ coordinate of the barycenter of 
the interior of the moment polygon $P$,  $\langle x_k \rangle_\imath$
is the $k^{\rm th}$ coordinate  of the center of the $\imath^{\rm th}$ 
edge of $P$, and ${\zap A}_\imath$ is the weighted length of $\imath^{\rm th}$ 
edge. 

We will now prove (\ref{norm}) using  a  method \cite{klp,ls} which 
is  broadly applicable to $\CC^\times$-actions, but which 
 nicely simplifies in  the toric case.  We thus  make a choice of 
$k=1$ or $2$, and set $\Xi=\Xi_k$, $\xi=\xi_k$,  and $x=x_k$
for this choice of $k$. In order to facilitate comparison with \cite{klp,mcp2,ls}, 
set  $\eta = \xi/2\pi$,
so that $\eta$ is a symplectic vector field of period $2\pi$, with Hamiltonian 
$t= x/2\pi$. 
 Let $\Sigma = M\sslash \CC^\times$ be the stable quotient
of $(M,J)$ by the action generated by $\Xi$, and observe that  the following
interesting special properties hold in our   toric setting:
\begin{itemize}
\item the stable quotient $\Sigma$ has genus $0$; and 
\item all the isolated $\CC^\times$ fixed points  
project to just two points $q_1,q_2\in\Sigma$.
\end{itemize}

Let ${\mathbf a},{\mathbf b}\in \RR$, respectively,  denote the  minimum and maximum 
of the Hamiltonian $t$, so that  $t(M)=[{\mathbf a},{\mathbf b}]$. 
If  $t^{-1}(\{{\mathbf a}\})$ or  $t^{-1}(\{{\mathbf b}\})$ is an isolated
fixed point, blow up  $M$ there to obtain $\hat{M}$,
and pull the metric $g$ back to $\hat{M}$ as a degenerate metric; otherwise, 
let $\hat{M}=M$. We then have a  holomorphic quotient map $\varpi: \hat{M}\to \Sigma$. 
Let $C_+$ and $C_-$ be the holormorphic curves in $\hat{M}$ given by 
$t^{-1}({\mathbf b})$ and $t^{-1}({\mathbf a})$, respectively. Except when they are just artifacts produced by
 blowing up, 
the curves 
$C_\pm$ number among the rational curves $C_\imath$ which project to the 
sides of the moment polygon $P$; the others, after proper transform if necessary, 
form a sub-collection   $\{ E_\jmath\} \subset \{ C_\imath \}$ 
characterized by $\varpi^{-1} (\{ q_1 , q_2\}) = \cup_\jmath E_\jmath$ for
a preferred pair of distinct points $q_1,q_2\in\Sigma$. 
Each $E_\jmath$ is the closure of  a $\CC^\times$-orbit, and we will let 
$m_\jmath\in \ZZ^+$ denote the order of the isotropy of  $\CC^\times$ acting on
the relevant orbit. 
Also let $t_\jmath^-$ and $t_\jmath^+$  denote the minimum and maximum of $t$
on $E_\jmath$, so that   $t(E_\jmath) = [t_\jmath^-, t_\jmath^+]$, and  observe that  
 $$\langle t \rangle_\jmath:= (t_\jmath^- + t_\jmath^+)/2$$
coincides with 
 the  average value of $t$ on $E_\jmath$ with respect to $g$-area measure.

Let us now define 
  $\wp : \hat{M}\to \Sigma\times [{\mathbf a},{\mathbf b}]$ to be the map  $\varpi \times t$. 
 If $p_1, \ldots , p_m$ are the images in 
 $\Sigma\times ({\mathbf a},{\mathbf b})$ of the isolated 
fixed points, and if 
$$X=[\Sigma\times ({\mathbf a},{\mathbf b})]-\{ p_1, \ldots , p_m\},$$
then the open dense set $Y=\wp^{-1}(X)\subset \hat{M}$
map be viewed as an orbifold $S^1$-principal bundle over $X$, and comes
equipped with a unique connection 1-form $\theta$ whose kernel is 
$g$-orthogonal to $\eta$ and which satisfies 
$\theta (\eta )=1$. 
We may now express the given  K\"ahler metric $g$ 
as 
$$ 
g=w\, \check{g}(t)+w\, dt^{\otimes 2}+w^{-1}\theta^{\otimes 2}~,
$$
for a positive functions $w>0$ on $X$
and a family  orbifold metrics $\check{g}(t)$ on
$\Sigma$.

Because 
$g$, $w$ and $dt$ are geometrically defined,  $\check{g}(t)$
 is an invariantly defined, $t$-dependent  orbifold K\"ahler metric
on  $\Sigma$ for all regular values of $t$; moreover, it is a 
smooth well-defined  tensor field  on all of $(\Sigma -\{ q_1, q_2\}) \times ({\mathbf a},{\mathbf b})$.  
Now notice that the K\"ahler quotient of $M$ associated with a regular value of the 
Hamiltonian is manifestly $(\Sigma , w\, \check{g}(t))$, and must therefore tend to 
the restriction of $g$ to $C_\pm$ as $t\to {\mathbf a}$ or ${\mathbf b}$. On the other hand,  $w^{-1}= g(\eta , \eta)$
by construction, and since $\eta$ is a Killing field of period $2\pi$ and Hamiltonian
$t$, we have $g(\eta, \eta) = 2|t-{\mathbf a}| + O(|t-{\mathbf a}|^2)$ 
near $t={\mathbf a}$, and similarly near 
$t={\mathbf b}$. 
Thus \cite{klp,ls}, letting  $\check{\omega} (t)$ be the K\"ahler form of $\check{g}(t)$,
 we  have 
 \begin{eqnarray*}
 \left.\check{\omega}\right|_{t={\mathbf a}} = \left.\check{\omega}\right|_{t={\mathbf b}}&=&0 \\ 
\left.\frac{d }{d t} \check{\omega} \right|_{t={\mathbf a}}  &=&
   \left.\hphantom{-}2\omega \right|_{C^-}  \\ 
\left.\frac{d }{d t} \check{\omega} \right|_{t={\mathbf b}}&=&
 \left. - 2\omega \right|_{C^+}~ .  
\end{eqnarray*}
 More surprisingly,  the calculations underlying the hyperbolic ansatz of 
 \cite{mcp2} show \cite[equation (3.16)]{ls} that  
 the  scalar curvature density of $g$ may be 
globally expressed on $Y\subset M$ as 
$$ s\vol = \left[2\check{\rho} - \frac{d^2}{dt^2} \check{\omega}\right]\wedge 
dt\wedge \theta$$
where $\check{\rho}(t)$ is the  Ricci form of $\check{g}(t)$. 
However, 
for regular values of $t\in ({\mathbf a},{\mathbf b})$, the Gauss-Bonnet formula for orbifolds 
tells us that 
\begin{eqnarray*}
\frac{1}{2\pi}\int_{\Sigma} \check{\rho}(t)&=& \chi (\Sigma ) -
\sum_{\jmath} \delta_\jmath(t)(1-\frac{1}{m_{\jmath}})\\
&=& \chi (S^2) - 2 + \sum_\jmath \frac{1}{m_{\jmath}}\delta_\jmath(t) \\
&=&  \sum_\jmath \frac{1}{m_{\jmath}}\delta_\jmath(t) 
\end{eqnarray*}
where  we have  introduced the characteristic function 
 $$\delta_{\jmath}(t)=\left\{ \begin{array}{ll}
1&t^{-}_{\jmath}<t<t^{+}_{\jmath}\\
0&\mbox{otherwise }\end{array}\right.$$
of $(t_\jmath^-,t_\jmath^+)$ in order 
to keep track of which two curves $E_\jmath$ meet a given regular level-set
 of the Hamiltonian function $t$.

Now the Futaki invariant is defined in terms of  the $L^2$ inner product of
the scalar curvature $s$ of $g$ with  normalized holomorphy potentials. 
It is therefore 
pertinent  to observe that 
\bea \int_{M} ts\vol &=& \int_Y ts\vol
\\&=& 
\int_Y t\left[2\check{\rho} - \frac{d^2}{dt^2} 
\check{\omega}\right]\wedge dt\wedge \theta
\\&=& 4\pi \int_{\mathbf a}^{\mathbf b}t\left[\int_{\Sigma}
\check{\rho}\right]~dt
-2\pi\int_{\Sigma}\left[\int_{\mathbf a}^{\mathbf b}t
\frac{d^2}{dt^2} \check{\omega}\right]~dt 
\\&=&4\pi \int_{\mathbf a}^{\mathbf b}2\pi\left[
 \sum_\jmath \frac{1}{m_{\jmath}}\delta_\jmath(t)
\right]t~dt
-2\pi \int_{\Sigma}\left( \left[t\frac{d}{dt} \check{\omega}
\right]^{\mathbf b}_{{\mathbf a}}-\int_{\mathbf a}^{\mathbf b}
\frac{d\check{\omega}}{dt} dt\right) 
\\&=&4\pi \sum_{\jmath}\frac{2\pi}{m_{\jmath}}\int_{t_\jmath^-}^{t_\jmath^+}
t~dt~
-2\pi \int_{\Sigma}\left( -2{\mathbf b}\omega\Big|_{t={\mathbf b}} -2{\mathbf a}\omega\Big|_{t={\mathbf a}} -
\left[\check{\omega}\right]_{\mathbf a}^{\mathbf b}\right) 
\\&=&4\pi \sum_{\jmath} 
\frac{2\pi({t_\jmath^+}-{t_\jmath^-})}{m_{\jmath}}~\frac{{t_\jmath^+}+{t_\jmath^-}}{2}~
+4\pi \Big( {\mathbf a}~ [\omega ] \cdot C^- + {\mathbf b} ~ [\omega ] \cdot C^+
\Big)
\\&=&4\pi \sum_{\jmath} \left([\omega ] \cdot E_\jmath\right)
\langle t \rangle_\jmath~
+4\pi \Big( {\mathbf a} ~[\omega ] \cdot C^- + {\mathbf b}~ [\omega ] \cdot C^+
\Big) 
\\&=&4\pi \sum_{\imath}
\langle t \rangle_\imath~{\zap A}_\imath = 2 \sum_{\imath}
\langle x \rangle_\imath~{\zap A}_\imath
\eea
where ${\zap A}_\imath = [\omega ] \cdot C_\imath$ is once again the area
of $C_\imath$. 
Since the holomorphy potential  of the holomorphic vector field
$\Xi$ is  $x=2\pi t$, we therefore have 
\begin{eqnarray}
-{\mathfrak F}(\Xi,[\omega]) &=& \int_M s(x-\bar{x}) d\mu \nonumber  \\
 &=&2\pi \int_M st~ d\mu - \bar{x}\int_M s ~d\mu  \nonumber\\
  &=& \Big( 4\pi \sum_{\imath}
\langle x \rangle_\imath~{\zap A}_\imath
\Big) - \bar{x}\Big(4\pi c_1\cdot [\omega ]\Big) \label{moe} 
\end{eqnarray}
where $\bar{x}$  again denotes the average value of  $x$ on $M$.

Next, notice that $\cup_\imath   C_\imath$ is the zero locus of the holomorphic section 
$\Xi_1 \wedge \Xi_2$ of the anti-conical line-bundle $K^{-1} = \wedge^2 T^{1,0}$,
and that, 
since the imaginary parts of $\Xi_1$ and $\Xi_2$ are Killing fields,
 this section is transverse to the zero section away from the intersection points
$C_\imath \cap C_\jmath$. 
 It follows that  the homology class of  $\cup_\imath   C_\imath$
is 
Poincar\'e dual to $c_1(M, J) = c_1 (K^{-1})$. 
Hence 
$$
c_1\cdot [\omega ] = \sum_\imath C_\imath \cdot [\omega ] =  \sum_\imath  {\zap A}_\imath
$$
so that  (\ref{moe})
simplifies to become
$${\mathfrak F}(\Xi,[\omega]) = -4\pi \sum_{\imath}
\Big(\langle x \rangle_\imath-\bar{x}\, \Big)~{\zap A}_\imath
$$
and   (\ref{norm}) therefore follows  by setting  $\Xi=\Xi_k$ and $x=x_k$. 
\end{proof}

With this preparation, we can  now calculate ${\mathcal A}(\Omega)$ for any toric surface. 

\begin{proofchiaro}
Relative to the basis given by the normalized holomorphy potentials 
$\{ x_k -\bar{x}_k ~|~ k=1,2\}$,
Proposition \ref{stage} tells us  that 
the restriction of the Futaki invariant to $\mathfrak t$  is given by
$$\vec{\mathfrak F}= ({\mathfrak F}_1, {\mathfrak F}_2) = -4\pi \, |\partial P| \, \vec{\mathfrak D}.$$
Since the $L^2$ inner product ${\mathbb B}_{\mathbf T}$ on $\mathfrak t$ 
 is given in this basis
by the moment-of-inertia matrix
$$\Pi =  \left[\int_P (x_j - \bar{x}_j) (x_k - \bar{x}_k) \, d{\zap a}\right] =  
\left[\int_M (x_j - \bar{x}_j) (x_k - \bar{x}_k) \, d\mu \right] ,$$
 the dual inner product ${\mathbb B}_{\mathbf T}^{-1}$ on ${\mathfrak t}^*$
 is represented  by the inverse matrix $\Pi^{-1}$, and Proposition \ref{bump} 
therefore tells us that 
$$\| \mathfrak F\|^2 = \vec{\mathfrak F} \cdot \Pi^{-1} \vec{\mathfrak F} = 
16 \pi^2 \, |\partial P|^2\, \vec{\mathfrak D} \cdot \Pi^{-1} \vec{\mathfrak D}.$$
Since  the first Chern class is Poincar\'e dual to the homology class of $\cup C_\imath$,  
$$c_1\cdot [\omega ] = \sum_\imath C_\imath\cdot [\omega ] = 
\sum_\imath {\zap A}_\imath  = |\partial P|,$$
while  $M$ has volume $|P|= [\omega ]^2/2$. Thus 
$${\mathcal A}(\Omega ) =  \frac{(c_1\cdot [\omega ])^2}{[\omega ]^2} + 
\frac{1}{32\pi^2} \| \mathfrak F\|^2 = \frac{|\partial P|^2}{2}\left(  \frac{1}{|P|} + 
 \vec{\mathfrak D} \cdot \Pi^{-1} \vec{\mathfrak D}\right)$$
exactly as  claimed. 
\end{proofchiaro}

By (\ref{better}), 
Theorem \ref{behold}  is now an  immediately immediate corollary.

\section{The Abreu Formalism}
\label{abreu}

The proof of Theorem \ref{behold} given in \S \ref{focus} was based on the author's 
 formula \cite{mcp2} for the scalar curvature
of K\"ahler surfaces with isometric $S^1$ actions. This section will present 
a different proof,  which is based on Abreu's beautiful formula \cite{abreueq}
for the scalar curvature
of a toric manifold, and makes crucial use of  an integration-by-parts trick 
 due to 
 Donaldson \cite{dontor}.  While this second proof is certainly more elegant and 
  natural, 
   there are unfortunately many numerical factors involved
in this formalism that are typically  misreported in the literature, and we will need 
to correct these imprecisions in order 
to obtain our result.  This will be well worth the effort, however, insofar as this second proof 
  works equally well in all complex dimensions.
The reader should note, however,  that the  higher-dimensional version of
Theorem \ref{behold} 
is of  much less differential-geometric interest  than the corresponding statement
in  complex dimension $2$;
 it is only in real  dimension $4$ that the Calabi energy is intimately tied to 
the Weyl functional and conformally Einstein metrics. 

We thus begin by considering a toric manifold $(M^{2m},J,\mathbf{T})$
of complex dimension $m$, equipped with a K\"ahler metric $g$ which is
invariant under the action of the $m$-torus $\mathbf{T}\cong T^m$. Choosing
an isomorphism 
$\mathbf{T} \cong  \RR^m/\ZZ^m$, we then let 
$(\xi_1, \ldots , \xi_m)$ be the $m$ unit-period  vector fields  
generating $\mathbf{T}$  associated with this choice, 
 and let $(\Xi_1, \ldots , \Xi_m)$ be the holomorphic vector fields 
defined by $\Xi_j = \xi_j^{1,0}$. Let 
$(x_1, \ldots , x_m)$
be Hamiltonians for $(\xi_1, \ldots , \xi_m)$, and note that these  are consequently also 
 holomorphy potentials for 
$(\Xi_1, \ldots , \Xi_m)$. 
The function $\vec{x}: M\to \RR^m$ given by  $(x_1, \ldots , x_m)$ is then a
 {\em moment map} for this $T^m$-action, and its image $\vec{x}(M)$ is 
 called the associated {\em moment polytope}. 
 Once again, the moment  polytope  has the 
  {\em Delzant property}:   a neighborhood of any vertex  $\in P$ 
can be 
transformed into a neighborhood of  $\vec{0}\in [0,\infty )^m$ 
by an element of $\mathbf{SL}(m, \ZZ) \ltimes \RR^m$. The $2m$-dimensional 
volume measure on $M$ now pushes forward, by integration on the fibers, to the standard
$m$-dimensional Euclidean measure on $\RR^m$, which we will again 
denote by $d{\zap a}$  to emphasize our special interest in the case of $m=2$. The boundary
$\partial P$ is the image of a union of toric complex hypersurfaces in $M$, and the 
push-forward of $(2m-2)$-dimensional Riemannian measure induces an $(m-1)$-dimensional measure $d\lambda$ 
on $\partial P$ which, on each face, is $\mathbf{SL}(n, \ZZ)$-equivalent to the standard $(m-1)$-dimensional 
Euclidean measure on the hyper-plane $x_1=0$.

For consistency with \cite{abreueq,dontor}, it will  be convenient to also consider
the vector fields $\eta_j = \xi_j /2\pi$ of period $2\pi$, and their Hamiltonians 
$t^j= x_j /2\pi$; the corresponding moment map is then $\vec{t} = (t^1, \ldots , t^m )$, 
and its image $\tilde{P}= \vec{t}(M)$ can then be transformed into $P$ by dilating
by a factor of $2\pi$. Following Donaldson, we will  use  $d\umu$ 
  to denote $m$-dimensional Euclidean measure
on $\tilde{P}$, and $d\sigma$ to denote the $(m-1)$-dimensional measure
on $\partial \tilde{P}$ which, on each face, is $\mathbf{SL}(n, \ZZ)$-equivalent to
$(m-1)$-dimensional Euclidean measure on the hyperplane $t^1=0$. Identifying 
$\tilde{P}$ with $P$ via the obvious homothety, we thus have $d{\zap a} = (2\pi)^m d\umu$
and $d\lambda = (2\pi)^{m-1} d\sigma$. 

On the open dense set $\vec{t}^{-1}(\Int \tilde{P}) \subset M$, 
Abreu observed that our  $T^m$-invariant K\"ahler metric 
can be expressed as 
$$g = V_{,jk} dt^j \otimes dt^k + V^{,jk}d\vartheta_j \otimes d\vartheta_k$$
where $V: \tilde{P} \to \RR$ is  a convex potential function, 
 $[V_{,jk}]$ is the Hessian matrix of $V$, $[V^{,jk}]$ is its inverse matrix,
and the $\vartheta_j$ are standard angle coordinates on $T^m = S^1 \times \cdots \times S^1$. 
The potential $V$ is Legendre dual to a K\"ahler potential for $g$; it is 
continuous on $\tilde{P}$ and smooth in its interior.  Moreover, 
it satisfies the so-called Guillemin-Abreu boundary condition:  near a face given by 
$L=0$, where the affine linear function $L: \RR^m\to \RR$ is non-negative on $\tilde{P}$ and where
$dL$ is 
 an indivisible
  element of the integer lattice $(\ZZ^m)^*$,  $V$ differs from $\frac{1}{2} L \log L$ by a
smooth function. (Note that the factor of $1/2$ is missing from \cite[p. 303]{dontor}, 
and will lead to a compensating correction below.) 
The scalar curvature $s$ of $g$ is then expressible in terms of $V$ via 
Abreu's beautiful  formula \cite{abreueq,dontor}
\begin{equation}\label{abra}
s = - (V^{,jk})_{,jk}: = -\sum_{j,k=1}^m \frac{\partial^2V^{,jk}}{\partial t_j \partial t_k} ,
\end{equation}
where we have followed 
 Donaldson's conventions  in order to   give $s$ 
its standard Riemannian value. 
  
  In this setting, Donaldson \cite[Lemma 3.3.5]{dontor} derives the integration-by-parts formula
  \begin{equation}
\label{cadabra}
 \int_{\tilde{P}}V^{,jk}f_{,jk} ~d\umu =  \int_{\tilde{P}}(V^{,jk})_{,jk}f ~d\umu + 2 \int_{\partial \tilde{P}} f ~d\sigma
\end{equation}
for any convex function $f$. Note, however, that the factor of $2$ in front of
the boundary term  does not actually  appear in \cite{dontor}, but  is needed to compensate
for the factor of $1/2$ in the corrected  Abreu-Guillemin boundary conditions. We also give 
 the boundary term a different sign, because 
we are treating $d\sigma$ as a measure rather than as an exterior differential form.

\begin{xpl}
Let $(M,g)$ be the unit $2$-sphere, with sectional curvature $K=1$, and hence with 
scalar curvature $s=2K=2$. Equip $(M,g)$ with the $S^1$ action given by 
period-$2\pi$ rotation around the the $z$-axis, with Hamiltonian $t=z$ and
moment polytope $\tilde{P}= [-1,1]$. In cylindrical coordinates,  
our metric  becomes 
$$
g= \frac{dt^2}{1-t^2} + (1-t^2) d\vartheta^2 
$$
so that the potential $V$ must satisfy $V_{,11}= 1/(1-t^2)$ and $V^{,11}= 1-t^2$. 
A suitable choice of $V$ is therefore 
$$V= \frac{1}{2} (1+t) \log (1+t) + \frac{1}{2} (1-t) \log (1-t)$$
and we note that this satisfies the  Guillemin-Abreu 
boundary conditions  discussed above. 
The Abreu formula \eqref{abra} now correctly calculates the scalar curvature
$$
s= -(V^{,11})_{,11} = -\frac{d^2}{dt^2} (1-t^2) =2 
$$
of $g$.  Also notice that integration by parts gives
$$
\int_{-1}^1 (1-t^2) f^{\prime\prime} dt = \int_{-1}^1 (1-t^2)^{\prime\prime} f dt + 2[f(-1)+f(1)]
$$
as predicted by   \eqref{cadabra}. 
\end{xpl}

\begin{xpl}
Let $(M^{2m},g)$ be the Riemannian product $S^2 \times \cdots \times S^2$ of $m$ copies
of the unit $2$-sphere, with equipped with the product $T^m$-action. The 
moment polytope is now the $m$-cube $\tilde{P} = [-1,1]^m$, and 
the metric is again represented by a symplectic 
potential 
$$
V= \frac{1}{2} \sum_j  \left[(1+t^j) \log (1+t^j) +  (1-t^j) \log (1-t^j)\right]
$$
which  satisfies our corrected Guillemin-Abreu boundary conditions. 
The Abreu formula \eqref{abra}  now predicts that the scalar curvature of $g$ is 
$$
s= -(V^{,ij})_{,ij} = -\sum_j \frac{\partial^2}{\partial(t^j)^2} (1-t_j^2) =2m~,
$$
in agreement with the additivity of the scalar curvature under Riemannian products. 
Integrating the $j^{\rm th}$  term  by parts twice  in the $j^{\rm th}$ variable, we  have 
$$
\int_{\tilde{P}}
\left[\sum_j [1-(t^j)^2] \frac{\partial^2f}{\partial (t^j)^2} \right] d\umu  = \int_{\tilde{P}}
\left[\sum_j   \frac{\partial^2[1-(t^j)^2] }{\partial (t^j)^2} \right] f~d\umu + 2 \int_{\partial \tilde{P}}f~d\sigma ~,
$$
for any smooth $f$, thereby double-checking    \eqref{cadabra} in  complex dimension $m$. 
\end{xpl}

 By linearity, \eqref{cadabra} also holds  \cite[Corollary 3.3.10]{dontor}
if $f$ is any difference of convex functions. In particular, 
 \eqref{cadabra}  applies to 
any affine linear 
function $f$ on  $\RR^m$; and since any such $f$ satisfies $f_{,jk}=0$, 
\eqref{abra} and \eqref{cadabra} tell us that 
$$
0= \int_{\tilde{P}}(-s) f ~d\umu + 2 \int_{\partial \tilde{P}} f ~d\sigma
$$
for any affine-linear function. Applying the dilation that relates $\tilde{P}$ and $P$, 
we therefore obtain
\begin{equation}
\label{kazam}
  \int_P  sf ~d{\zap a} = 4\pi   \int_{\partial P} f ~d\lambda 
\end{equation}
for any affine-linear  $f$. In particular, if we take $f=x_k - \bar{x}_k$,
we obtain 
$$ \int_P  x_k  (s-\bar{s}) ~d{\zap a}= \int_P  (x_k - \bar{x}_k) s ~d{\zap a} =
4\pi  \int_{\partial P} (x_k - \bar{x}_k)  ~d\lambda$$
which in turn implies  that 
 $$ \int_M x_k  (s-\bar{s}) ~d\mu =4\pi  \int_{\partial P} (x_k - \bar{x}_k)  d\lambda $$
 because $d{\zap a}$ is the push-forward of the volume measure of $(M,g)$. However, 
 $x_k$ is a holomorphy potential  for the holomorphic vector field $\Xi_k$, so  
   \eqref{foot} tells us that the component  
      $$ \mathfrak{F}_k  := \mathfrak{F} (\Xi_k , \Omega)  $$
      of the Futaki invariant 
 is given by 
   $$ \mathfrak{F}_k = -4\pi  \int_{\partial P} (x_k - \bar{x}_k)  
 ~d\lambda ~. $$
On the other hand, 
$$\frac{1}{|\partial P|}\int_{\partial P} (x_k - \bar{x}_k)  d\lambda = \langle x_k - \bar{x}_k \rangle 
= \langle x_k \rangle - \bar{x}_k = \mathfrak{D}_k $$
where  $|\partial P|$ denotes the $\lambda$-measure of the boundary, 
$\langle ~\rangle$ is the average with respect to $d\lambda$, and where 
$\mathfrak{D}_k$  is the $k^{\rm th}$ component of the 
vector $\mathfrak{D}$ which  points from  the barycenter of $P$ to the 
barycenter of $\partial P$. 
Thus the Futaki invariant ${\mathfrak{F}}(\Omega ) =\vec{\mathfrak F}= ({\mathfrak F}_1 , \ldots , {\mathfrak F}_m)$
is given by 
\begin{equation}
\label{pointer}
\vec{\mathfrak F}  = -4\pi~ |\partial P| ~\vec{\mathfrak{D}}
\end{equation}
and we have thus reproved Proposition \ref{stage}  in arbitrary complex dimension 
$m$.

Now notice that, by taking normalized Hamiltonians,  
the Lie algebra $\mathfrak{t}$ of our maximal torus $\mathbf{T}$ 
is naturally  identified with those affine-linear functions $\RR^m\to \RR$
which send the barycenter $\bar{\vec{x}}$ of our moment polytope to 
$0$. From this view-point,  it is now 
apparent that  ${\mathfrak{F}}(\Omega ) = -4\pi |\partial P|~\vec{\mathfrak{D}}$ 
actually belongs to $\mathfrak{t}^*$, as it should. In these same terms, though, the  
``moment-of inertia'' matrix
$\Pi$ defined by 
\begin{equation}
\label{inertia}
\Pi_{jk} = \int_P (x_j - \bar{x}_j) (x_k - \bar{x}_k)~d{\zap a} 
\end{equation}
 represents the $L^2$ inner product 
$$
\mathbb{B}_{\mathbf{T}}: \mathfrak {t} \times \mathfrak {t}  \to \RR ~,
$$
while its inverse  matrix $\Pi^{-1}$  represents the dual inner product
$$
\mathbb{B}_{\mathbf{T}}^{-1}: \mathfrak {t}^* \times \mathfrak {t}^*  \to \RR~.
$$
By Corollary \ref{pump} and \eqref{pointer}, we thus have 
 $$\| {\mathfrak F}(\Omega )\|^2 = {\mathbb B}_{\mathbf T}^{-1}\Big(   {\mathfrak F}(\Omega ) \, ,  \, {\mathfrak F}(\Omega )\Big)= 16\pi^2 |\partial P|^2 ~\vec{\mathfrak{D}} \cdot \Pi^{-1} \vec{\mathfrak{D}} ~. $$
Chen's inequality  \eqref{shine} therefore tells us\footnote{Here it is worth  reiterating  that,  
while the inequality \eqref{shine}
is  essentially elementary when $g$ is $\mathbf{T}$-invariant, it is a  deep  and remarkable
 result  that 
this same inequality  in fact  holds 
for  completely arbitrary K\"ahler metrics.}  that any K\"ahler metric on a toric manifold
satisfies 
$$
\int_M (s-\bar{s})^2 d\mu \geq {16\pi^2} |\partial P|^2 
~\vec{\mathfrak{D}} \cdot \Pi^{-1} \vec{\mathfrak{D}}
$$
where the moment polytope $P$ is determined solely by the toric manifold $M$ 
and the K\"ahler class 
$\Omega$; moreover, equality holds iff $g$ is extremal. 

On the other hand, setting $f=1$ in \eqref{kazam} yields 
$$\int_P s ~d{\zap a} = 4\pi \int_{\partial P}  d\lambda ~,$$ 
so that 
$$\int_M s~d\mu = 4\pi ~|\partial P |~,$$
a fact which the reader may enjoy comparing with \eqref{ints}. 
Since $(M,g)$ has volume $|P|$, we therefore see that 
$$
\int_M \bar{s}^2 ~d\mu = \frac{\left(\int_M s~d\mu \right)^2}{\int_Md\mu}= 16\pi^2 \frac{|\partial P|^2}{|P|}
$$
and the Pythagorean theorem \eqref{python} therefore implies the following result:
\begin{main}
\label{behind}
 Let $(M^{2m}, J, \Omega,\mathbf{T})$ be a toric complex $m$-manifold with fixed K\"ahler class, and let
 $P\subset \RR^m$ be the associated moment polytope. Then the scalar curvature $s$  of 
 any K\"ahler metric $g$ with K\"ahler form $\omega \in \Omega$  
satisfies 
\begin{equation}
\label{dejavu}
\frac{1}{16\pi^2}\int_M s^2 d\mu_g \geq 
|\partial P|^2 \Big( \frac{1}{|P|~} + \vec{\mathfrak D}\cdot  {\Pi}^{-1} \vec{\mathfrak D}
\Big)~,
\end{equation}
with 
equality iff $g$ is an extremal K\"ahler metric. Here $|P|$ denotes the $m$-volume of 
the interior of  $P$, $|\partial P|$ is the $\lambda$-volume of its boundary,  
 the  moment-of-inertia
matrix  $\Pi$ of $P$ is defined by \eqref{inertia}, and $\vec{\mathfrak D}$ is the vector joining  the barycenter
$P$ to  the barycenter of $\partial P$. 
\end{main}

Specializing to the case of $m=2$ gives a second proof of Theorem \ref{behold}.

Notice that the sharp lower bound 
\eqref{dejavu} is in fact independent of dimension. However, this feature of the result 
actually depends
on our conventions regarding the moment polytope and the generators of the action. 
For example, if we had instead chosen the periodicity of our generators
to be $2\pi$ instead of $1$, we would have been led to instead use the   polytope $\tilde{P}$, 
and we would have then been forced to introduce 
an inconvenient
scaling factor, since  
$$
\frac{|\partial {P}|^2}{|{P}|} = (2\pi)^{m-2}\frac{|\partial \tilde{P}|^2}{|\tilde{P}|} 
$$
But it is also worth noticing that 
this awkward scaling factor magically disappears when $m=2$. This reflects
the fact that the Calabi energy is invariant under rescaling in  real dimension four, 
and that  rescaling a K\"ahler class exactly results in a rescaling of the associated 
moment polytope.

In particular, for the purpose of calculating the virtual action $\mathcal{A}$ 
for toric surfaces, we would have 
obtained exactly the same formula if we had used  the rescaled polygon
 $\tilde{P}$ instead of the polygon  $P$  emphasized by this article. 
Nonetheless, the use of $P$ has other practical advantages, even when $m=2$.
For example,  the $\lambda$-length  of sides of $P$ directly represents the areas of holomorphic
curves in $M$, unmediated by factors of $2\pi$. In practice,
this avoids repeatedly having to  cancel powers
of $2\pi$ when calculating $\mathcal{A}(\Omega)$ in  explicit examples. 
This  will   now become apparent, as we next 
 illustrate  Theorem \ref{behold} by applying it to specific   
 toric surfaces. 

\section{Hirzebruch Surfaces} 

As a simple illustration of  Theorem \ref{chiaro}, we now compute ${\mathcal A}(\Omega)$
for the {Hirzebruch surfaces}. Recall \cite{bpv,gh} that, for any non-negative integer $k$, 
 the $k^{\rm th}$ Hirzebruch surface $\mathbb{F}_k$ 
is defined to be the $\CP_1$-bundle ${\mathbb P}({\mathcal O}(k)\oplus {\mathcal O})$
over $\CP_1$; that is, it is the complex surface  obtained from   line bundle
${\mathcal O}(k)\to \CP_1$ of  Chern class $k$
by adding a section at infinity.  
Calabi \cite{calabix} explicitly constructed an extremal  every K\"ahler metric in 
every K\"ahler class on each $\mathbb{F}_k$;  his direct assault on the problem proved
 feasible because 
the maximal compact subgroup  ${\mathbf U}(2)/\ZZ_k$ of the automorphism group has orbits 
of real codimension 1, thereby reducing the relevant equation for the K\"ahler potential to an
 ODE. Because their automorphism groups all  contain finite quotients of  ${\mathbf U}(2)$,
the Hirzebruch surfaces all admit actions of the $2$-torus  
$T^2$, and so are toric surfaces.  
Normalizing the fibers of  $\mathbb{F}_k\to \CP_1$ to have area 1, the associated
moment polygon  becomes the trapezoid 
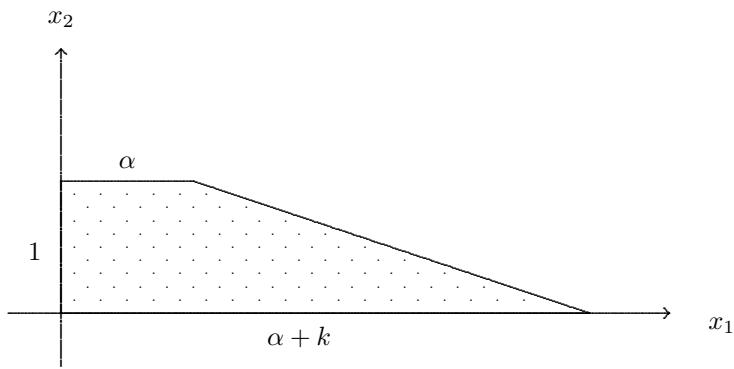
\begin{center}
\mbox{
\beginpicture
\setplotarea x from 0 to 180, y from 0 to 150 
\arrow <2pt> [1,2] from 0 20  to 250 20
\arrow <2pt> [1,2] from 20 0 to 20 120 
\put {$x_1$} [B1] at 270 14  
\put {$x_2$} [B1] at  20 130  
\put {$1$} [B1] at 10 40  
\put {$\alpha$} [B1] at 45 75 
\put {$\alpha+k$} [B1] at 110 8  
\hshade 20 20  220 70  20 70    /  
{\setlinear
\plot 20 20 20 70 /
\plot 20 20  220 20 / 
\plot 20 70 70 70 /
\plot 70 70 220 20 / 
}
\endpicture
}
\end{center}
and since ${\mathcal A}(\Omega)$ is unchanged by  multiplying $\Omega$ by a positive constant, 
we may impose this normalization without loss of generality.

We will now apply  Theorem \ref{chiaro} to calculate the Calabi energy of Calabi's extremal 
K\"ahler metrics; 
 since Hwang and Simanca \cite{hwasim} have previously computed  this quantity  by  
other means,  this exercise will, among other things, 
 provide us with another useful double-check of equation (\ref{sounder}). 
The area and $\lambda$-perimeter of the polygon are easily seen to be 
$$|P|= \alpha + \frac{k}{2}, ~~|\partial P| = 2+2\alpha + k$$
and it is not difficult to calculate the barycenter of the interior 
$$\bar{\vec{x}}= \frac{\left(3\alpha^2 + 3k \alpha + k^2,  3\alpha + k\right)
}{6|P|}$$
or boundary 
$$   
\langle \vec{x}\rangle =\frac{ \left(
\alpha^2+  \alpha (k+1) + \frac{1}{2}k(k +1)  , 
\alpha+1
\right)}{|\partial P |} 
$$
by hand. The vector 
$$
\vec{\mathfrak D}=
{\textstyle \frac{k(2\alpha+k-1)}{12|\partial P | |P|} }
 \Big( k ,  -2 \Big)
$$
thus joins these two barycenter, and without too much work
one can also check  that the ``moment-of inertia'' matrix of $P$ is given by 
$$
\Pi = \frac{1}{72 |P|}\left[\begin{array}{cc} 
\textstyle {6 \alpha^4 + 
12 \alpha^3k  +12\alpha^2k^2 + 6 \alpha k^3 + k^4 }
 &  -\frac{k}{2}(6\alpha^2+6\alpha k +k^2) 
 \\     -\frac{k}{2}(6\alpha^2+6\alpha k +k^2) & 
 6 \alpha^2 +6 \alpha k  + k^2
\end{array}\right]
$$
The Futaki contribution to $\mathcal A$  is therefore encoded  by the expression 
$$
\vec{\mathfrak D}\cdot \Pi^{-1}\vec{\mathfrak D}= \frac{2k^2(2\alpha+k-1)^2}{|P||\partial P|^2(6\alpha^2+6\alpha k+ k^2)}
$$
and the virtual action is thus given by 
\begin{equation}
\label{hirzen}
{\mathcal A}(\Omega) =  \frac{2\alpha^3+(4+3k)\alpha^2+2(1+k)^2\alpha+k(k^2+2)/2}{\alpha^2+\alpha k+ k^2/6}~.
\end{equation}
After
multiplication by an overall constant and 
 the
change of variables $k=n$,  $\alpha= (a-n)/2$, 
this agrees with 
 with the expression 
Hwang and Simanca \cite[equation (3.2)]{hwasim} obtained for   their ``potential energy'' 
via a  different method.

For $k> 0$, the function ${\mathcal A}(\alpha)$ on the right-hand side of
(\ref{hirzen}) extends  smoothly 
across $\alpha=0$, and  satisfies 
$$\left. \frac{d{\mathcal A}}{d\alpha} \right|_{\alpha =0} = -6\frac{(k-2)^2}{k}$$
so ${\mathcal A}(\alpha)$ is  a decreasing function for small $\alpha$ if $k\neq 2$.
On the other hand, ${\mathcal A}(\alpha) \sim 2\alpha$ for $\alpha \gg 0$, so ${\mathcal A}$ is increasing
for large $\alpha$. It follows that ${\mathcal A}(\alpha)$ has a minimum somewhere on $\RR^+$
for any $k\neq 2$. Since Calabi's construction \cite{calabix} moreover shows that each 
K\"ahler class on a Hirzebruch surface is represented by an extremal K\"ahler metric, 
Proposition \ref{bfk} tells us that, for $k\neq 2$,  
the Calabi metric $g_k$ 
corresponding to the minimizing value of $\alpha$  is necessarily Bach-flat.

On the other hand,  since 
$$
{\mathcal A}(\Omega) -\frac{3}{4}k = \frac{48\alpha^3+(54k+96)\alpha^2+(30k^2+96k+48)\alpha + 9k^3+24k }{4(6\alpha^2+6k\alpha + k^2)}
$$
is positive for all $\alpha > 0$,  it follows that 
$$\min_\Omega {\mathcal A}(\Omega) > \frac{3}{4}k, $$
and we conclude that the corresponding Bach-flat K\"ahler metric $g_k$ has 
$${\mathcal W}(g_k) > 2\pi^2k
$$
Since the  Hirzebruch surface $\mathbb{F}_k$ is diffeomorphic to
$S^2\times S^2$ when $k$ is even, and  is diffeomorphic to 
$\CP_2\#\overline{\CP}_2$ when $k$ is odd,    the metrics $g_k$, first discovered by 
Hwang and Simanca  \cite{hwasim}, immediately give us the following: 
 
\begin{prop}
The smooth $4$-manifolds $S^2 \times S^2$ and $\CP_2\#\overline{\CP}_2$
both admit sequences of Bach-flat conformal classes $[g_{k_j}]$ with  ${\mathcal W}([g_{k_j}])\to +\infty$. Consequently, the moduli space of Bach-flat conformal metrics on either of 
these manifolds has infinitely many connected components. 
\end{prop}

The metric $g_1$ on $\mathbb{F}_1$ has scalar curvature $s > 0$ everywhere, 
and its conformal rescaling $s^{-2}g_1$ was shown by Derdzi{\'n}ski \cite{derd} to coincide with 
the Einstein metric on $\CP_2\#\overline{\CP}_2$ discovered by Page 
\cite{page}. For $k \geq 3$, the scalar curvature $s$ of $g_k$ instead vanishes along  a
hypersurface, which becomes the conformal infinity for the Einstein metric
$s^{-2}g_k$; thus $\mathbb{F}_k$ is obtained from two Poincar\'e-Einstein
manifolds, glued along their conformal infinity. These two Einstein metrics
are in fact isometric, in an orientation-reversing manner. Because of their ${\mathbf U}(2)$
symmetry, these Einstein metrics 
belong to the
family first discovered by B\'erard-Bergery \cite{beber3}, and later  rediscovered by physicists, who call them  AdS-Taub-bolt metrics \cite{hawknut}.

\section{The Two-Point Blow-Up of $\CP_2$}

As a final illustration of Theorem \ref{chiaro}, we now compute the virtual action
for K\"ahler classes on the blow-up of  $\CP_2$ at two distinct points. 
The
present author has done this elsewhere by a more complicated method, 
and the details of the answer played an important role in showing \cite{chenlebweb,lebhem10} that 
this manifold admits an Einstein metric, obtained by conformally rescaling
a Bach-flat K\"ahler metric. Thus, repeating the computation by means 
of equation (\ref{sounder}) provides yet another double-check of Theorem \ref{behold}. 

Blowing up $\CP_2$ in two distinct points results in exactly the same complex surface 
as blowing  $\CP_1\times \CP_1$ in a single point  \cite{bpv,gh}. The latter picture is actually 
useful in choosing a  pair of generators for the torus action which makes the needed
computations as simple as possible. The resulting moment polygon $P$ then takes the form
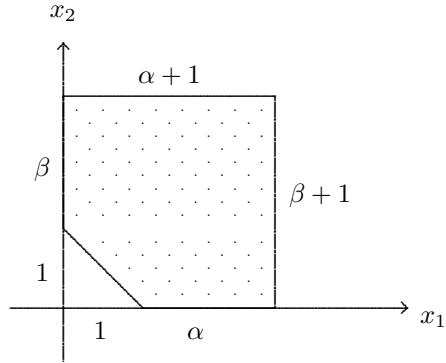
\begin{center}
\mbox{
\beginpicture
\setplotarea x from 0 to 200, y from 0 to 150 
\arrow <2pt> [1,2] from 0 20  to 150 20
\arrow <2pt> [1,2] from 20 0 to 20 120 
\put {$x_1$} [B1] at 160 14  
\put {$x_2$} [B1] at  20 130  
\put {$\beta$} [B1] at 12 70  
\put {$1$} [B1] at 12 30  
\put {$1$} [B1] at 34 8  
\put {$\alpha$} [B1] at 70 8 
\put {$\alpha+1$} [B1] at 60 105  
\put {$\beta+1$} [B1] at 117 60  
\hshade 20 50  100 50  20 100  100 20 100  /  
{\setlinear
\plot 20 50 20 100 /
\plot 50 20  100  20 / 
\plot 20 50  50 20 /
\plot 20 100 100 100 /
\plot 100 100 100 20 / 
}
\endpicture
}
\end{center}
after rescaling to give the blow-up divisor area $1$. It is then easy to see that
the area of the polygon and the $\lambda$-length of its boundary are given by
$$
|P| = \frac{1}{2} + \alpha + \beta + \alpha\beta , \quad  
|\partial P| = 3 + 2 \alpha + 2 \beta
$$
while the barycenter of the boundary 
$$
\langle \vec{x}\rangle = \frac{\Big(
(1 + \alpha) (2 + \alpha + \beta) , \quad(1 + \beta) (2 + \alpha + \beta) \Big)}{|\partial P|}
$$
and of the interior 
$$
\bar{\vec{x}}= 
\frac{
\Big( 
3(1+\alpha)^2(1+\beta)-1 ,  \quad 3(1+\alpha)(1+\beta)^2-1
 \Big) }{6\, |P|}
$$
are not difficult to compute by hand. The vector joining these two barycenters  is thus given by 
$$
\vec{\mathfrak D}= 
\frac{\Big(
-\alpha + 2 \beta + 3 \alpha \beta + 3 \alpha^2 \beta, \quad - \beta +2 \alpha + 3 \alpha \beta + 3 \alpha \beta^2
\Big) }{6 |P|~|\partial P | }
$$
 and  the moment-of-inertia matrix 
 $$
 \Pi = \frac{1}{24}
 \left[\begin{array}{cc}8(1+\alpha)^3(1+\beta)-2 & 6(1+\alpha)^2(1+\beta)^2-1 \\6(1+\alpha)^2(1+\beta)^2-1 & 8(1+\alpha)(1+\beta)^3-2\end{array}\right] -\, |P| \,\left[\begin{array}{cc}\bar{x}_1^2 & \bar{x}_1\bar{x}_2 \\ \bar{x}_1\bar{x}_2 & \bar{x}_1^2\end{array}\right]
$$
are also easily obtained without the use of a computer.  According to  (\ref{sounder}), 
 ${\mathcal A}(\Omega )$ is therefore given by 
 
 \bigskip 
 
\noindent {\small
$\displaystyle  
3 \Big[ 3 + 28 \beta + 96 \beta^2 + 168 \beta^3 + 164 \beta^4 + 80 \beta^5 + 16 \beta^6 + 
     16 \alpha^6 (1 + \beta)^4 + 
     16 \alpha^5 (5 + 24 \beta + 43 \beta^2 + 37 \beta^3 + 15 \beta^4 + 2 \beta^5) + 
     4 \alpha^4 (41 + 228 \beta + 478 \beta^2 + 496 \beta^3 + 263 \beta^4 + 60 \beta^5 + 
        4 \beta^6) + 
     8 \alpha^3 (21 + 135 \beta + 326 \beta^2 + 392 \beta^3 + 248 \beta^4 + 74 \beta^5 + 
        8 \beta^6) + 
     4 \alpha (7 + 58 \beta + 176 \beta^2 + 270 \beta^3 + 228 \beta^4 + 96 \beta^5 + 16 \beta^6) + 
     4 \alpha^2 (24 + 176 \beta + 479 \beta^2 + 652 \beta^3 + 478 \beta^4 + 172 \beta^5 + 
        24 \beta^6)\Big]\Big/\linebreak 
        \Big[1 + 10 \beta + 36 \beta^2 + 64 \beta^3 + 60 \beta^4 + 24 \beta^5 + 
   24 \alpha^5 (1 + \beta)^5 + 12 \alpha^4 (1 + \beta)^2 (5 + 20 \beta + 23 \beta^2 + 10 \beta^3) + 
   16 \alpha^3 (4 + 28 \beta + 72 \beta^2 + 90 \beta^3 + 57 \beta^4 + 15 \beta^5) + 
   12 \alpha^2 (3 + 24 \beta + 69 \beta^2 + 96 \beta^3 + 68 \beta^4 + 20 \beta^5) + 
   2 \alpha (5 + 45 \beta + 144 \beta^2 + 224 \beta^3 + 180 \beta^4 + 60 \beta^5)\Big]$}

\bigskip 

\noindent
as is most easily checked at this point using {\em Mathematica} or a similar program. After the substitution 
$\gamma=\alpha$, 
this agrees exactly with the answer obtained in \cite[\S 2]{lebuniq}, where this explicit 
formula 
plays a key role in classifying compact Einstein $4$-manifolds for which the 
metric is Hermitian with respect to some complex structure.

When $\alpha=\beta$, the above expression simplifies  to become 

$$\frac{9 + 96 \alpha + 396 \alpha^2 + 840 \alpha^3 + 954 \alpha^4 + 528 \alpha^5 + 96 \alpha^6}{1 + 
 12 \alpha + 54 \alpha^2 + 120 \alpha^3 + 138 \alpha^4 + 72 \alpha^5 + 12 \alpha^6}$$
which, after dividing by $3$ and making the substitution $\alpha=1/y$, coincides with
the expression \cite{lebhem} first used to show that $\mathcal A$ has a critical point, and 
later used again  \cite{chenlebweb} to prove the existence of a 
conformally Einstein, K\"ahler metric on $\CP_2\# 2\overline{\CP}_2$. For a second, 
conceptually simpler
  proof of this last fact, see \cite{lebhem10}. 


\appendix

\section{Restricting the Futaki Invariant}
\label{conceptual} 

 In this appendix, we will  prove Proposition \ref{sauce}. The key ingredient  used in the proof is 
 the following result of  Nakagawa \cite{nakagawa1}:
 
 \begin{prop}[Nakagawa] \label{naka} 
 Let $(M,J)$ be a projective algebraic complex manifold, let ${\mathbf H}$ be the identity component of its complex automorphism group, and suppose that the Jacobi homomorphism
 from ${\mathbf H}$ to the Albanese torus of $M$ is trivial. Let $L\to M$ be an ample line bundle
for which the action of\,  ${\mathbf H}$ on $M$ lifts to an action on 
 $L\to M$, and let $\Omega$ be the K\"ahler class defined by $\Omega= c_1(L)$. 
 Then  the Futaki invariant
 ${\mathfrak F}(\Omega)\in {\mathfrak h}^*$ annihilates the 
 Lie algebra ${\mathfrak r}_{\mathfrak u}$ of the unipotent radical of ${\mathbf H}$.
 \end{prop}
 This generalizes  a previous result of Mabuchi \cite{mabuchifoot} concerning
the case when $L$ is the anti-canonical line bundle. Both of
these results are proved using Tian's localization formula \cite{tianfoot}
for the Futaki invariant 
of a Hodge metric. 
 
We will now  extend Proposition \ref{naka} to irrational K\"ahler classes
 on certain complex manifolds.  In order to do this, we will first need the following 
 observation: 

  \begin{lem}\label{initio}
  Let $(M,J)$ be a compact complex manifold with $b_1(M)=0$, and let 
  ${\mathbf H}$ be the identity component of its complex automorphism group. 
  If $L\to M$ is a positive line bundle, then the action of ${\mathbf H}$ on $M$
  lifts to an action on  $L^{k}\to M$ for some positive
 integer $k$. 
  \end{lem}
  \begin{proof}
  By the Kodaira embedding theorem \cite{gh}, $L$ has a positive power
  $L^\ell$ for which there is a canonical 
   holomorphic embedding $j: M\hookrightarrow {\mathbb P}({\mathbb V})$
   such that $j^*{\mathcal O}(-1) = L^{-\ell}$, where   
   ${\mathbb V}:=[H^0(M, {\mathcal O}(L^\ell))]^*$.
   
  Now since $(M,J)$ is  of K\"ahler type and $H^1(M, \CC)=0$, the Hodge decomposition 
  tells us that $H^{0,1}(M)=H^1(M, {\mathcal O})=0$,  and the long exact
  sequence 
  $$\cdots \to H^1(M, {\mathcal O})\to H^1(M, {\mathcal O}^\times)\to H^2 (M, \ZZ)\to \cdots$$
  therefore implies that holomorphic line bundles on $M$ are classified by their first Chern classes. 
  On the other hand, since ${\mathbf H}$ is connected, each automorphism $\Phi : M\to M$,  
  $\Phi\in {\mathbf H}$, is homotopic to the identity; and since Chern classes are homotopy invariants, 
 we deduce that   that  $c_1(\Phi^*L)= c_1(L)$ for all $\Phi\in {\mathbf H}$. 
 Consequently, $\Phi^*L\cong L$ as a holomorphic
  line bundle for any $\Phi\in {\mathbf H}$. While the resulting isomorphism $\Phi^*L\cong L$ is
  not unique, any two such isomorphisms merely differ by an overall multiplicative constant, 
  and the  associated linear map $H^0(M, {\mathcal O}(L^\ell))\to H^0(M, {\mathcal O}(L^\ell))$
  induced by $\Phi^*$ 
 is therefore completely determined up to an overall scale factor. Thus, for
 every $\Phi\in {\mathbf H}$, there is a uniquely determined projective transformation
  ${\mathbb P}({\mathbb V})\to
 {\mathbb P}({\mathbb V})$, where again ${\mathbb V}:=[H^0(M, {\mathcal O}(L^\ell))]^*$.
This gives us  a faithful 
projective representation  ${\mathbf H}\hookrightarrow {\mathbf P\mathbf  S \mathbf L} ({\mathbb V})$ 
which acts on $M\subset {\mathbb P}({\mathbb V})$ via the original action 
of ${\mathbf H}$. 
 
Now consider the group ${\mathbf  S \mathbf L} ({\mathbb V})$   of unit-determinant 
linear endomorphisms of ${\mathbb V}$,
and observe that there is a short  exact sequence
$$
0\to \ZZ_n \to {\mathbf  S \mathbf L} ({\mathbb V})\to {\mathbf P\mathbf  S \mathbf L} ({\mathbb V})\to 1
$$
where  $n=\dim {\mathbb V}$; that is, every projective transformation of 
${\mathbb P}({\mathbb V})$ arises from $n$ different linear unit-determinant 
linear endomorphisms of ${\mathbb V}$, differing from each other  merely by multiplication by 
an  $n^{\rm th}$ root of unity. If $\widetilde{\mathbf H}< \mathbf{SL} ({\mathbb V})$ 
is the inverse
image of ${\mathbf H} <  \mathbf{PSL} ({\mathbb V})$, 
 then $\widetilde{{\mathbf H}}$ acts on ${\mathbb V}$,
and so also acts on the tautological  line bundle ${\mathcal O}(-1)$
over ${\mathbb P}({\mathbb V})$. Restricting ${\mathcal O}(-1)$ 
to $M$ then gives us an
action of $\widetilde{{\mathbf H}}$ on $L^{-\ell}$ which lifts the action of ${\mathbf H}$
on $M$, in such a manner that  any two lifts of a given element
only differ by multiplication of an $n^{\rm th}$ root of unity. The induced action
of $\widetilde{{\mathbf H}}$ on $L^{-n\ell}$ therefore descends to an action of ${\mathbf H}$, 
 and passing to the dual line bundle  $L^{n\ell}$ thus 
shows that the action of ${\mathbf H}$ on $M$ can be lifted to an action on $L^k\to M$
for $k=n\ell$. 
  \end{proof}
  
\begin{prop} \label{reduce} 
Let $(M,J)$ be a compact complex manifold of K\"ahler type, 
and suppose that $M$ does not carry any non-trivial holomorphic $1$- or $2$-forms. 
Then, for any K\"ahler class $\Omega$ on $M$, the Futaki invariant 
${\mathfrak F}(\Omega) \in {\mathfrak h}^*$ annihilates the unipotent radical
${\mathfrak r}_{\mathfrak u}\subset {\mathfrak h}$.\end{prop}
\begin{proof}
By hypothesis, $H^{1,0}(M)=H^{2,0}(M)=0$. The Hodge decomposition therefore
tells us that $b_1(M)=0$ and that $H^{1,1}(M,\RR) = H^2(M, \RR)$. Consequently,
 the K\"ahler cone ${\zap K}\subset H^{1,1}(M,\RR)$
is  open in $H^2(M, \RR)$. Since 
$H^2(M, \QQ )$ is dense in $H^{2}(M,\RR)$, it follows that 
 $H^2(M, \QQ )\cap {\zap K}$ is dense in ${\zap K}$. In particular, 
 $H^2(M, \QQ )\cap {\zap K}$  is non-empty, and so, clearing denominators,
 we conclude that the K\"ahler cone ${\zap K}$ must meet the 
 the integer lattice $H^2(M, \ZZ)/\mbox{torsion}\subset H^2 (M, \RR)$. 
 This argument, due to Kodaira \cite{kodrr},  shows that 
 $(M,J)$ carries K\"ahler metrics of Hodge type,  and  
is therefore projective algebraic. 

Pursuing this idea in the opposite direction, let 
$\Psi$ now be an integral K\"ahler class, so that $\Psi =c_1(L)$ for some
 positive line bundle $L\to M$. By Lemma \ref{initio}, the action of $\mathbf{H}$
on $M$ then lifts to some positive power $L^k$ of $L$. Since our hypotheses
also imply that the Albanese torus is trivial, Proposition \ref{naka} therefore 
implies  that ${\mathfrak F}(k\Psi)\in {\mathfrak h}^*$
 annihilates ${\mathfrak r}_{\mathfrak u}$. However, our expression (\ref{foot})
 for the Futaki invariant  implies that 
 $$
{\mathfrak F} ( \Xi  , \lambda \Omega)=\lambda^m {\mathfrak F} ( \Xi  , \Omega)
 $$
for any $\lambda\in \RR^+$, where $m$ is the complex dimension, since  rescaling 
a K\"ahler metric by $g \rightsquigarrow \lambda g$
results in $\omega \rightsquigarrow \lambda \omega$, 
 $s\rightsquigarrow \lambda^{-1}s$, $f \rightsquigarrow \lambda f$, and 
$d\mu \rightsquigarrow \lambda^m d\mu$. By taking  $\lambda$ to be an arbitrary positive 
rational, we therefore see that 
${\mathfrak F} ( \Xi  ,  \Omega)=0$ whenever $\Xi\in {\mathfrak r}_{\mathfrak u}$
and  $\Omega\in H^2(M,\QQ)\cap {\zap K}$, where 
${\zap K}$ once again denotes 
 the K\"ahler cone. However, for any fixed $\Xi$, the right-hand-side of (\ref{foot}) clearly
depends smoothly on the K\"ahler metric $g$,  and ${\mathfrak F} ( \Xi  ,  \Omega)$ therefore
is a smooth function of the K\"ahler class $\Omega$. 
But  $h^{2,0}(M)=0$ implies  that $H^2(M,\QQ)\cap {\zap K}$ is dense in 
${\zap K}$. Thus, for any  $\Xi\in {\mathfrak r}_{\mathfrak u}$, 
 we have shown that   ${\mathfrak F} ( \Xi  ,  \Omega)=0$ 
for a  dense set of $\Omega\in {\zap K}$. Continuity therefore 
implies that ${\mathfrak F} ( \Xi  ,  \Omega)=0$
for {\em all} $\Omega \in {\zap K}$. Hence ${\mathfrak F} (  \Omega)\in {\mathfrak h}^*$
annihilates  ${\mathfrak r}_{\mathfrak u}$ for any K\"ahler class $\Omega$ on $M$.\end{proof}

 Under the hypotheses of Proposition \ref{reduce}, we 
 can thus view ${\mathfrak F}(\Omega)$
 as belonging to the complexified Lie coalgebra ${\mathfrak g}_\CC^*={\mathfrak g}^*\otimes \CC$ 
 of a maximal compact subgroup ${\mathbf G}\subset {\mathbf H}$. By averaging, 
 let us now represent our given K\"ahler class $\Omega$ by a  ${\mathbf G}$-invariant 
  K\"ahler metric $g$.
 The Lie algebra of Killing fields of $g$ then can be identified with the real holomorphy
 potentials of integral $0$, which are their Hamiltonians; the Lie bracket on $\mathfrak g$
 then becomes the Poisson bracket $\{ \cdot , \cdot \}$ on Hamiltonians. Since 
 the scalar curvature $s$ of $g$ is also a real function, formula (\ref{foot}) thus
 tells us that ${\mathfrak F}(\Omega)$ is actually a {\em real} linear functional
 on ${\mathfrak g}$; that is, ${\mathfrak F}(\Omega)\in {\mathfrak g}^*$.
This proves Proposition \ref{sauce}.


\begin{thebibliography}{10}

\bibitem{abreueq}
{\sc M.~Abreu},
\newblock K\"ahler geometry of toric varieties and extremal metrics, {\em
  Internat. J. Math.} {\bf 9} (1998) 641--651.

\bibitem{arpasing}
{\sc C.~Arezzo}, {\sc F.~Pacard}, and {\sc M.~Singer},
\newblock Extremal metrics on blowups, {\em Duke Math. J.} {\bf 157} (2011)
  1--51.

\bibitem{atcvx}
{\sc M.~F. Atiyah},
\newblock Convexity and commuting {H}amiltonians, {\em Bull. London Math. Soc.}
  {\bf 14} (1982) 1--15.

\bibitem{aubin}
{\sc T.~Aubin},
\newblock \'{E}quations du type {M}onge-{A}mp\`ere sur les vari\'et\'es
  k\"ahleriennes compactes, {\em C. R. Acad. Sci. Paris S\'er. A-B} {\bf 283}
  (1976) Aiii, A119--A121.

\bibitem{bach}
{\sc R.~Bach},
\newblock {Zur Weylschen Relativit\"atstheorie und der Weylschen Erweiterung
  des Kr\"ummungstensorbegriffs.}, {\em Math. Zeitschr.} {\bf 9} (1921)
  110-135.

\bibitem{bando}
{\sc S.~Bando},
\newblock An obstruction for {C}hern class forms to be harmonic, {\em Kodai
  Math. J.} {\bf 29} (2006) 337--345.

\bibitem{bpv}
{\sc W.~Barth}, {\sc C.~Peters}, and {\sc A.~Van~de Ven},
\newblock {\em Compact complex surfaces}, volume~4 of {\em Ergebnisse der
  Mathematik und ihrer Grenzgebiete (3)},
\newblock Springer-Verlag, Berlin, 1984.

\bibitem{beber3}
{\sc L.~B{\'e}rard-Bergery},
\newblock Sur de nouvelles vari\'et\'es riemanniennes d'{E}instein,
\newblock in {\em Institut \'{E}lie {C}artan, 6}, volume~6 of {\em Inst. \'Elie
  Cartan}, pp. 1--60, Univ. Nancy, Nancy, 1982.

\bibitem{bes}
{\sc A.~L. Besse},
\newblock {\em Einstein manifolds}, volume~10 of {\em Ergebnisse der Mathematik
  und ihrer Grenzgebiete (3)},
\newblock Springer-Verlag, Berlin, 1987.

\bibitem{nick}
{\sc N.~Buchdahl},
\newblock On compact {K}\"ahler surfaces, {\em Ann. Inst. Fourier (Grenoble)}
  {\bf 49} (1999) 287--302.

\bibitem{calabix}
{\sc E.~Calabi},
\newblock Extremal {K}\"ahler metrics,
\newblock in {\em Seminar on Differential Geometry}, volume 102 of {\em Ann.
  Math. Studies}, pp. 259--290, Princeton Univ. Press, Princeton, N.J., 1982.

\bibitem{calabix2}
{\sc E.~Calabi},
\newblock Extremal {K}\"ahler metrics. {II},
\newblock in {\em Differential {G}eometry and {C}omplex {A}nalysis}, pp.
  95--114, Springer, Berlin, 1985.

\bibitem{xxel}
{\sc X.~X. Chen},
\newblock Space of {K}\"ahler metrics. {III}. {O}n the lower bound of the
  {C}alabi energy and geodesic distance, {\em Invent. Math.} {\bf 175} (2009)
  453--503.

\bibitem{xxgang2}
{\sc X.~X. Chen} and {\sc G.~Tian},
\newblock Geometry of {K}\"ahler metrics and foliations by holomorphic discs,
  {\em Publ. Math. Inst. Hautes \'Etudes Sci.}  (2008) 1--107.

\bibitem{chenlebweb}
{\sc X.~Chen}, {\sc C.~LeBrun}, and {\sc B.~Weber},
\newblock On conformally {K}\"ahler, {E}instein manifolds, {\em J. Amer. Math.
  Soc.} {\bf 21} (2008) 1137--1168.

\bibitem{chenlisheng}
{\sc B.~Chen}, {\sc A.-M. Li}, and {\sc L.~Sheng},
\newblock Extremal Metrics on Toric Surfaces,
\newblock e-print, arXiv:1008.2607v3 [math.DG], 2010.

\bibitem{chevalley}
{\sc C.~Chevalley},
\newblock {\em Th\'eorie des groupes de {L}ie. {T}ome {III}. {T}h\'eor\`emes
  g\'en\'eraux sur les alg\`ebres de {L}ie},
\newblock Actualit\'es Sci. Ind. no. 1226, Hermann \& Cie, Paris, 1955.

\bibitem{delzant}
{\sc T.~Delzant},
\newblock Hamiltoniens p\'eriodiques et images convexes de l'application
  moment, {\em Bull. Soc. Math. France} {\bf 116} (1988) 315--339.

\bibitem{derd}
{\sc A.~Derdzi{\'n}ski},
\newblock Self-dual {K}\"ahler manifolds and {E}instein manifolds of dimension
  four, {\em Compositio Math.} {\bf 49} (1983) 405--433.

\bibitem{donaldsonk1}
{\sc S.~K. Donaldson},
\newblock Scalar curvature and projective embeddings. {I}, {\em J. Differential
  Geom.} {\bf 59} (2001) 479--522.

\bibitem{dontor}
{\sc S.~K. Donaldson},
\newblock Scalar curvature and stability of toric varieties, {\em J.
  Differential Geom.} {\bf 62} (2002) 289--349.

\bibitem{fujikiaut}
{\sc A.~Fujiki},
\newblock On automorphism groups of compact {K}\"ahler manifolds, {\em Invent.
  Math.} {\bf 44} (1978) 225--258.

\bibitem{fultor}
{\sc W.~Fulton},
\newblock {\em Introduction to toric varieties}, volume 131 of {\em Annals of
  Mathematics Studies},
\newblock Princeton University Press, Princeton, NJ, 1993.

\bibitem{futaki0}
{\sc A.~Futaki},
\newblock An obstruction to the existence of {E}instein {K}\"ahler metrics,
  {\em Invent. Math.} {\bf 73} (1983) 437--443.

\bibitem{fuma1}
{\sc A.~Futaki} and {\sc T.~Mabuchi},
\newblock Uniqueness and periodicity of extremal {K}\"ahler vector fields,
\newblock in {\em Proceedings of GARC Workshop on Geometry and Topology '93
  (Seoul, 1993)}, volume~18 of {\em Lecture Notes Ser.}, pp. 217--239, Seoul,
  1993, Seoul Nat. Univ.

\bibitem{gh}
{\sc P.~Griffiths} and {\sc J.~Harris},
\newblock {\em Principles of Algebraic Geometry},
\newblock Wiley-Interscience, New York, 1978.

\bibitem{guiltor}
{\sc V.~Guillemin},
\newblock {\em Moment maps and combinatorial invariants of {H}amiltonian
  {$T^n$}-spaces}, volume 122 of {\em Progress in Mathematics},
\newblock Birkh\"auser Boston Inc., Boston, MA, 1994.

\bibitem{gscvx}
{\sc V.~Guillemin} and {\sc S.~Sternberg},
\newblock Convexity properties of the moment mapping, {\em Invent. Math.} {\bf
  67} (1982) 491--513.

\bibitem{hawknut}
{\sc S.~W. Hawking}, {\sc C.~J. Hunter}, and {\sc D.~N. Page},
\newblock N{UT} charge, anti-de {S}itter space, and entropy, {\em Phys. Rev. D
  (3)} {\bf 59} (1999) 044033, 6.

\bibitem{hwasim}
{\sc A.~D. Hwang} and {\sc S.~R. Simanca},
\newblock Extremal {K}\"ahler metrics on {H}irzebruch surfaces which are
  locally conformally equivalent to {E}instein metrics, {\em Math. Ann.} {\bf
  309} (1997) 97--106.

\bibitem{klp}
{\sc J.~Kim}, {\sc C.~LeBrun}, and {\sc M.~Pontecorvo},
\newblock Scalar-flat {K}\"ahler surfaces of all genera, {\em J. Reine Angew.
  Math.} {\bf 486} (1997) 69--95.

\bibitem{kobfixed}
{\sc S.~Kobayashi},
\newblock Fixed points of isometries, {\em Nagoya Math. J.} {\bf 13} (1958)
  63--68.

\bibitem{kodrr}
{\sc K.~Kodaira},
\newblock On compact complex analytic surfaces. {I}, {\em Ann. of Math. (2)}
  {\bf 71} (1960) 111--152.

\bibitem{mcp2}
{\sc C.~LeBrun},
\newblock Explicit self-dual metrics on {${\mathbb {C}}{\mathbb {P}}\sb
  2\#\cdots\#{\mathbb {C}}{\mathbb {P}}\sb 2$}, {\em J. Differential Geom.}
  {\bf 34} (1991) 223--253.

\bibitem{lebicm}
{\sc C.~LeBrun},
\newblock Anti-self-dual metrics and {K}\"ahler geometry,
\newblock in {\em Proceedings of the International Congress of Mathematicians,
  Vol.\ 1, 2 (Z\"urich, 1994)}, pp. 498--507, Basel, 1995, Birkh\"auser.

\bibitem{lebhem}
{\sc C.~LeBrun},
\newblock Einstein metrics on complex surfaces,
\newblock in {\em Geometry and {P}hysics ({A}arhus, 1995)}, volume 184 of {\em
  Lecture Notes in Pure and Appl. Math.}, pp. 167--176, Dekker, New York, 1997.

\bibitem{lebhem10}
{\sc C.~LeBrun},
\newblock Einstein Manifolds and Extremal {K\"ahler} Metrics,
\newblock to appear in Crelle; e-print arXiv:1009.1270 [math.DG], 2010.

\bibitem{lebuniq}
{\sc C.~LeBrun},
\newblock On {E}instein, {H}ermitian $4$-Manifolds,
\newblock e-print arXiv:1010.0238 [math.DG], 2010.

\bibitem{ls2}
{\sc C.~LeBrun} and {\sc S.~R. Simanca},
\newblock On the {K}\"ahler classes of extremal metrics,
\newblock in {\em Geometry and Global Analysis (Sendai, 1993)}, pp. 255--271,
  Tohoku Univ., Sendai, 1993.

\bibitem{ls}
{\sc C.~LeBrun} and {\sc S.~R. Simanca},
\newblock Extremal {K}\"ahler metrics and complex deformation theory, {\em
  Geom. Funct. Anal.} {\bf 4} (1994) 298--336.

\bibitem{mabary}
{\sc T.~Mabuchi},
\newblock Einstein-{K}\"ahler forms, {F}utaki invariants and convex geometry on
  toric {F}ano varieties, {\em Osaka J. Math.} {\bf 24} (1987) 705--737.

\bibitem{mabuchifoot}
{\sc T.~Mabuchi},
\newblock An algebraic character associated with the {P}oisson brackets,
\newblock in {\em Recent topics in differential and analytic geometry},
  volume~18 of {\em Adv. Stud. Pure Math.}, pp. 339--358, Academic Press,
  Boston, MA, 1990.

\bibitem{mabuniq}
{\sc T.~Mabuchi},
\newblock Uniqueness of extremal {K}\"ahler metrics for an integral {K}\"ahler
  class, {\em Internat. J. Math.} {\bf 15} (2004) 531--546.

\bibitem{milmo}
{\sc J.~Milnor},
\newblock {\em Morse {T}heory}, volume~51 of {\em Ann. Math. Studies},
\newblock Princeton University Press, Princeton, N.J., 1963,
\newblock Based on lecture notes by M. Spivak and R. Wells.

\bibitem{nakagawa1}
{\sc Y.~Nakagawa},
\newblock Bando-{C}alabi-{F}utaki characters of {K}\"ahler orbifolds, {\em
  Math. Ann.} {\bf 314} (1999) 369--380.

\bibitem{nakagawa2}
{\sc Y.~Nakagawa},
\newblock Bando-{C}alabi-{F}utaki character of compact toric manifolds, {\em
  Tohoku Math. J. (2)} {\bf 53} (2001) 479--490.

\bibitem{sunspot}
{\sc Y.~Odaka}, {\sc C.~Spotti}, and {\sc S.~Sun},
\newblock Compact Moduli Spaces of {D}el {P}ezzo Surfaces and
  {K}\"ahler-{E}instein Metrics,
\newblock e-print arXiv:1210.0858 [math.DG], 2012.

\bibitem{page}
{\sc D.~Page},
\newblock A Compact Rotating Gravitational Instanton, {\em Phys. Lett.} {\bf
  79B} (1979) 235--238.

\bibitem{shelukhin}
{\sc E.~Shelukhin},
\newblock Remarks on invariants of {H}amiltonian loops, {\em J. Topol. Anal.}
  {\bf 2} (2010) 277--325.

\bibitem{siu}
{\sc Y.~Siu},
\newblock Every {K3} Surface is {K\"a}hler, {\em Inv. Math.} {\bf 73} (1983)
  139--150.

\bibitem{tasd}
{\sc C.~H. Taubes},
\newblock The existence of anti-self-dual conformal structures, {\em J.
  Differential Geom.} {\bf 36} (1992) 163--253.

\bibitem{tian}
{\sc G.~Tian},
\newblock On {C}alabi's conjecture for complex surfaces with positive first
  {C}hern class, {\em Invent. Math.} {\bf 101} (1990) 101--172.

\bibitem{tianfoot}
{\sc G.~Tian},
\newblock K\"ahler-{E}instein metrics on algebraic manifolds,
\newblock in {\em Transcendental methods in algebraic geometry ({C}etraro,
  1994)}, volume 1646 of {\em Lecture Notes in Math.}, pp. 143--185, Springer,
  Berlin, 1996.

\bibitem{yauma}
{\sc S.~T. Yau},
\newblock On the {R}icci curvature of a compact {K}\"ahler manifold and the
  complex {M}onge-{A}mp\`ere equation. {I}, {\em Comm. Pure Appl. Math.} {\bf
  31} (1978) 339--411.

\end{thebibliography}
  \end{document}